\documentclass[12pt]{amsart}
\usepackage{amssymb}
\usepackage{mathrsfs}
\usepackage{graphicx}
\usepackage{color}
\renewcommand\a{\alpha}
\renewcommand\b{\beta}
\newcommand\g{\gamma}
\renewcommand\d{\delta}
\newcommand\la{\lambda}

\newcommand\s{\sigma}

\newcommand\vf{\varphi}

\renewcommand\t{\tau}

\newcommand\w{\omega}

\newcommand\vL{\varLambda}

\newcommand\vG{\varGamma}
\newcommand\ve{\varepsilon}

\newcommand{\QQ}{\mathbb Q}

\newcommand{\ZZ}{\mathbb Z}

\newcommand\BP{\mathbf P}

\newcommand\Bp{\mathbf p}
\newcommand\Bm{\mathbf m}

\newcommand\bx{\mathbf{x}}
\newcommand\bX{\mathbf{X}}

\newcommand\CA{\mathcal{A}}
\newcommand\CB{\mathcal{B}}

\newcommand\CL{\mathcal{L}}

\newcommand\CK{\mathcal{K}}

\newcommand\CF{\mathcal{F}}

\newcommand\CT{ \mathcal{T}}

\newcommand\FS{\mathfrak S}

\newcommand\Fg{\mathfrak g}

\newcommand\Fgl{\mathfrak{gl}}

\newcommand\wt{\widetilde}

\newcommand\ol{\overline}

\newcommand{\lan}{\langle}
\newcommand{\ran}{\rangle}

\newcommand{\ra}{\rightarrow }

\newcommand\Hom{\operatorname{Hom}}
\newcommand\End{\operatorname{End}}

\newcommand\Res{\operatorname{Res}}

\newcommand\ch{\operatorname{ch}}

\newcommand{\rad}{\operatorname{rad}}

\newcommand{\cmod}{\operatorname{-mod}}

\newcommand{\LR}{\operatorname{LR}}

\newcommand{\isom}{\,\raise2pt\hbox{$\underrightarrow{\sim}$}\,}
\newcounter{ichi}
\newcommand{\roi}{\roman{ichi}}
\newcounter{ni}
\newcommand{\roii}{\roman{ni}}
\newcounter{san}
\newcommand{\roiii}{\roman{san}}
\newcounter{yon}
\newcommand{\roiv}{\roman{yon}}
\newcounter{go}
\newcounter{roku}
\newcounter{nana}
\newcommand{\Sc}{\mathscr{S}}

\newcommand{\He}{\mathscr{H}}

\newcommand{\bo}[1]{
\begin{array}{|c|}
\hline 
#1
\\
\hline
\end{array}
}

\newtheorem{thm}{Theorem}[section]
\newtheorem{lem}[thm]{Lemma}
\newtheorem{cor}[thm]{Corollary}
\newtheorem{prop}[thm]{Proposition}

\def \para{\refstepcounter{thm} \par\medskip\noindent
                \textbf{\thethm .} }

\def \remark{\refstepcounter{thm} \par\medskip\noindent
                \textbf{Remark \thethm .} }

\def \remarks{\refstepcounter{thm} \par\medskip\noindent
                \textbf{Remarks \thethm .} }

\def \example{\refstepcounter{thm} \par\medskip\noindent
                \textbf{Example \thethm .} }

\allowdisplaybreaks[4]
\numberwithin{equation}{thm}

\begin{document}
\setlength{\baselineskip}{4.9mm}
\setlength{\abovedisplayskip}{4.5mm}
\setlength{\belowdisplayskip}{4.5mm}
\renewcommand{\theenumi}{\roman{enumi}}
\renewcommand{\labelenumi}{(\theenumi)}
\renewcommand{\thefootnote}{\fnsymbol{footnote}}
\renewcommand{\thefootnote}{\fnsymbol{footnote}}
\parindent=20pt
\newcommand{\dis}{\displaystyle}

\medskip
\begin{center}
{\large \bf On  Weyl modules of cyclotomic $q$-Schur algebras} 
\\
\vspace{1cm}
Kentaro Wada\footnote{This research was supported by GCOE \lq Fostering top leaders in mathematics', Kyoto University.}

\address{Research Institute for Mathematical Sciences, Kyoto University, Kyoto 606-8502, JAPAN}
\email{wada@kurims.kyoto-u.ac.jp} 

\end{center}

\title{}
\maketitle 
\markboth{Kentaro Wada}
{On Weyl modules of cyclotomic $q$-Schur algebras}




\begin{abstract}
We study on  Weyl modules of cyclotomic $q$-Schur algebras. 
In particular, 
we give the character formula of the Weyl modules 
by using the Kostka numbers 
and 
some numbers which are computed by a generalization of Littlewood-Richardson rule. 
We also study corresponding symmetric functions. 
Finally, 
we give some simple applications to modular representations of cyclotomic $q$-Schur algebras. 
\end{abstract}


\setcounter{section}{-1}

\section{Introduction} 
Let
$_R \He_{n,r}$ be the Ariki-Koike algebra over a commutative ring $R$ 
with parameters $q,Q_1,\cdots,Q_r \in R$ 
associated to the complex reflection group 
$\FS_{n} \ltimes (\ZZ/r \ZZ)^n$, 
and 
let 
$_R \Sc_{n,r}$ be the cyclotomic $q$-Schur algebra 
associated to $_R \He_{n,r}$ 
introduced by \cite{DJM98}. 
Put $\CA=\ZZ[q,q^{-1}, Q_1,\cdots, Q_r]$,  
where $q,Q_1,\cdots, Q_r$ are indeterminate,    
and 
$\CK =\QQ(q,Q_1,\cdots,Q_r)$ 
is the quotient field of $\CA$ 
(In this introduction, we omit the subscript $\CK$ when we consider an algebra over $\CK$). 

In the case where $r=1$, 
$_R \He_{n,1}$ is the Iwahori-Hecke algebra associated to the symmetric group $\FS_n$, 
and 
$_R \Sc_{n,1}$ is the $q$-Schur algebra associated to $_R \He_{n,1}$.  
In this case, 
the $q$-Schur algebra $_R \Sc_{n,1}$ comes from the Schur-Weyl duality between 
$_R \He_{n,1}$ and the quantum group $_R U_q(\Fgl_m)$ as follows. 
Let $\Fgl_m$ be the general linear Lie algebra,  
and $U_q(\Fgl_m)$ be the corresponding quantum group over $\CK$. 
We consider the vector representation $V$ of $U_q(\Fgl_m)$, 
then $U_q(\Fgl_m)$ acts on the tensor space $V^{\otimes n}$ via coproduct of $U_q(\Fgl_m)$. 
$\He_{n,1}$ also acts on the tensor space $V^{\otimes n}$ 
by a $q$-analogue of the permutations for the ingredient of the tensor product. 
Then the Schur-Weyl duality between $U_q(\Fgl_m)$ and $\He_{n,1}$ holds via this tensor space $V^{\otimes n}$ 
by \cite{J}. 
Moreover, this Schur-Weyl duality also holds even over  $\CA$ (see \cite{Du}). 
Hence, we can specialize to any ring $R$ with a parameter $q \in R^\times$. 
Then 
the $q$-Schur algebra $_R \Sc_{n,1}$ coincides with the image of  $_R U_q(\Fgl_m) \ra \End(V^{\otimes n})$ 
which comes from the action of $_RU_q(\Fgl_m)$ on $V^{\otimes n}$.

On the other hand, 
in the case where $r \geq 2$, 
it is also known the Schur-Weyl duality by \cite{SakS}. 
Let $\Fg=\Fgl_{m_1} \oplus \cdots \oplus \Fgl_{m_r}$ 
be a Levi subalgebra of $\Fgl_m$, 
and 
$U_q(\Fg)$ be the corresponding quantum group over $\CK$. 
$U_q(\Fg)$ acts on $V^{\otimes n}$ by the restriction of the action of $U_q(\Fgl_m)$.  
We can also define the action of $\He_{n,r}$ on $V^{\otimes n}$ 
which is a generalization of the action of $\He_{n,1}$. 
Then 
$U_q(\Fg)$ and $\He_{n,r}$ satisfy the Schur-Weyl duality via the tensor space $V^{\otimes n}$ 
by \cite{SakS}. 
Unfortunately, 
this Schur-Weyl duality does not hold over $\CA$ 
since the action of $\He_{n,r}$ on $V^{\otimes n}$ 
is not defined over $\CA$. 
However, 
we can replace $\He_{n,r}$ with the modified Ariki-Koike algebra $_R \He_{n,r}^0$ over $R$ with parameters $q,Q_1,\cdots,Q_r$ 
such that $Q_i - Q_j$ ($i \not=j$) is invertible in $R$ 
which was introduced by \cite{Sho}. 
Then, 
the Schur-Weyl duality 
between $_R U_q(\Fg)$ and $_R \He_{n,r}^0$ holds via the tensor space $V^{\otimes n}$ (see \cite{SS}). 
Let $_R \Sc_{n,r}^0$ be the image of  $_R U_q(\Fg) \ra \End(V^{\otimes n})$ 
which comes from the action of $_RU_q(\Fg)$ on $V^{\otimes n}$. 
Then 
some relations between 
$_R \Sc_{n,r}^0$ and $_R \Sc_{n,r}$ 
are studied in \cite{SS} and \cite{Saw}. 
In particular, 
$_R \Sc_{n,r}^0$ is realized as a subquotient algebra of $_R \Sc_{n,r}$.  
Then, some decomposition numbers of $_R \Sc_{n,r}$ coincide with the decomposition numbers of $_R \Sc_{n,r}^0$ 
(which are also decomposition number for $_R U_q(\Fg)$) when $R$ is a field. 
In \cite{SW}, 
we obtained a certain generalization of  
these results (see also Remark \ref{remark product formula}). 
Motivated by this generalization together with the Schur-Weyl duality 
between $_R U_q(\Fg)$ and $_R \He_{n,r}^0$, 
the author gave a presentation of $\Sc_{n,r}$ (also $_\CA \Sc_{n,r}$) 
by generators and fundamental relations in \cite{W}. 
By using this presentation, 
we can define a (not surjective) homomorphism 
$\Phi_\Fg : U_q(\Fg) \ra \Sc_{n,r}$. 
We also have $\Phi_{\Fg}|_{_\CA U_q(\Fg)} : \, _\CA U_q(\Fg) \ra \, _\CA \Sc_{n,r}$ by restriction. 
Thus we can specialize it to any commutative ring $R$ and parameters $q,Q_1.\cdots,Q_r \in R$.  
In this paper, 
we study $_R \Sc_{n,r}$-modules by restricting the action to $_R U_q(\Fg)$ when $R$ is a field. 

First, 
we consider over $\CK$. 
In this case, 
$\Sc_{n,r}$ is semi-simple, 
and 
finite dimensional $U_q(\Fg)$-module is also semi-simple. 
Put $\vL_{n,r}^+=\{\la=(\la^{(1)}, \cdots,\la^{(r)}) \,|\, \la^{(k)}$: partition, $\sum_{k=1}^r |\la^{(k)}|=n\}$, 
the set of $r$-partitions of size $n$. 
Let   
$W(\la)$ 
be the Weyl module of $\Sc_{n,r}$ corresponding to $\la \in \vL_{n,r}^+$. 
It is well known that 
$\{W(\la)\,|\, \la \in \vL_{n,r}^+\}$ 
gives a complete set of non-isomorphic simple $\Sc_{n,r}$-modules. 
By investigating the appearing weights, 
we see that 
$\{W(\la^{(1)}) \boxtimes \cdots \boxtimes W(\la^{(r)}) \,|\, \la \in \vL_{n,r}^+\}$ 
gives a complete set of non-isomorphic simple $U_q(\Fg)$-modules 
which appear as $U_q(\Fg)$-submodules of  $\Sc_{n,r}$-modules through the homomorphism $\Phi_{\Fg}$, 
where 
$W(\la^{(k)})$ is the Weyl module of $U_q(\Fgl_{m_k})$ with the highest weight $\la^{(k)}$.
Then we can consider the irreducible decomposition 
of the Weyl module $W(\la)$ of $\Sc_{n,r}$ as $U_q(\Fg)$-modules 
through the homomorphism $\Phi_{\Fg}$ as follows:  
\[
\tag{0.1}  
W(\la) \cong \bigoplus_{\mu \in \vL_{n,r}^+} \Big( W(\mu^{(1)}) \boxtimes \cdots \boxtimes W(\mu^{(r)}) \Big)^{\oplus \b_{\la\mu}} 
\text{ as $U_q(\Fg)$-modules} .
\] 
In order to compute the multiplicity $\b_{\la\mu}$ in this decomposition, 
we describe the $U_q(\Fg)$-crystal structure on $W(\la)$ 
by using a generalization of \lq\lq admissible reading" for $U_q(\Fgl_m)$-crystal given in \cite{KN} 
(Theorem \ref{Theorem b_la_mu}). 
As a consequence, 
we can compute the multiplicity $\b_{\la\mu}$ by the combinatorial way 
which can be regarded as a generalization of the Littlewood-Richardson rule 
(Corollary \ref{Corollary b_la_mu}. See also Remark \ref{Remark LR rule}). 

Thanks to the decomposition (0.1), 
we obtain the character formula of $W(\la)$
by using  Kostka numbers and multiplicities $\b_{\la\mu}$ ($\la,\mu \in \vL_{n,r}^+$) 
(Note that the weight space as the $\Sc_{n,r}$-module coincides with the weight space as the $U_q(\Fg)$-module 
from the homomorphism $\Phi_{\Fg}$).  
We also describe the  character of $W(\la)$ 
as a linear combination of the products of Schur polynomials with coefficients $\b_{\la\mu}$ ($\la,\mu \in \vL_{n,r}^+$). 
Moreover, 
we see that 
the set of characters of the Weyl modules for all $r$-partitions 
gives a new basis of the ring of symmetric polynomials 
(Theorem \ref{Theorem character of Weyl}). 
Then we also study on some properties for such symmetric functions.

As an application of  the decomposition (0.1), 
we have a certain factorization of decomposition matrix of $_R \Sc_{n,r}$ when $R$ is a field 
(Theorem \ref{Theorem factorization decom matrix}), 
and we give an alternative proof of the product formula for decomposition numbers of $_R \Sc_{n,r}$ 
given by \cite{Saw} (Corollary \ref{Corollary prod formula}, See also Remark \ref{remark product formula}.). 
For special parameters ($Q_1=\cdots=Q_r=0$ or $q=1$, $Q_1=\cdots =Q_r$), 
we can compute the decomposition matrix of $_R \Sc_{n,r}$ 
from the  factorization of decomposition matrix (Corollary \ref{Corollary Q_1=Q_2= cdots =Q_r=0}).

Finally, 
we realize the Ariki-Koike algebra $_R \He_{n,r}$ 
as a subalgebra of $_R \Sc_{n,r}$ by using the generators of $_R \Sc_{n,r}$ 
(Proposition \ref{Proposition AK in Sc}), 
and we give an alternative proof for the classification of simple $\He_{n,r}$-modules 
for special parameters ($Q_1=\cdots=Q_r=0$ or $q=1$, $Q_1=\cdots =Q_r$) 
which has already obtained by \cite{AM} and \cite{M} (Corollary \ref{Corollary simple AK}). 
\\

\textbf{ Acknowledgments :}
The author is grateful to Professors 
S. Ariki, H. Miyachi and T. Shoji for 
many valuable discussions and comments.



\section{Review on cyclotomic $q$-Schur algebras} 
In this section, we recall the definition of the cyclotomic $q$-Schur algebra $\Sc_{n,r}$ introduced by \cite{DJM98}, 
and 
we review  presentations of $\Sc_{n,r}$ by generators and fundamental relations given by \cite{W}.

\para 
\label{Definition Ariki-Koike}
Let $R$ be a commutative ring,  
and 
we take parameters 
$q,Q_1,\cdots, Q_r \in R$ 
such that 
$q$ is invertible in $R$. 
The Ariki-Koike algebra $_R \He_{n,r}$ associated to 
the complex reflection group  
$\FS_n \ltimes (\ZZ/r\ZZ)^n$ 
is the associative algebra with  1 over $R$ 
generated by 
$T_0,T_1,\cdots,T_{n-1}$ 
with the following defining relations: 
\begin{align*}
&(T_0-Q_1)(T_0-Q_2)\cdots (T_0-Q_r)=0, \\
&(T_i-q)(T_i+q^{-1})=0 &&(1 \leq i \leq n-1),\\
&T_0 T_1 T_0 T_1=T_1 T_0 T_1 T_0, \\
&T_i T_{i+1} T_i = T_{i+1} T_i T_{i+1} &&(1\leq i \leq n-2),\\
&T_i T_j=T_j T_i &&(|i-j|\geq 2). 
\end{align*}

The subalgebra of 
$_R \He_{n,r}$ 
generated by 
$T_1,\cdots,T_{n-1}$ 
is isomorphic to 
the Iwahori-Hecke algebra 
$_R \He_n$ 
of the symmetric group  $\FS_n$ of degree $n$.  
For 
$w \in \FS_n$, 
we denote by 
$\ell(w)$ the length of $w$, 
and denote by 
$T_w$  the standard basis of $_R \He_n$ corresponding to $w$.

\para 
Let $\Bm =(m_1,\cdots, m_r) \in \ZZ_{>0}^r$ 
be an $r$-tuple of positive integers. 
Put
\[ \vL_{n,r} (\Bm) =\left\{ \mu=(\mu^{(1)},\cdots,\mu^{(r)}) \Biggm| 
	\begin{matrix}
		\mu^{(k)}=(\mu_1^{(k)},\cdots,\mu_{m_k}^{(k)}) \in \ZZ_{\geq 0}^{m_k} \\
	    \sum_{k=1}^r \sum_{i=1}^{m_k} \mu_i^{(k)}=n 
	\end{matrix}	
	\right\}. 
\]
We denote by $|\mu^{(k)}|= \sum_{i=1}^{m_k} \mu_i^{(k)}$ (resp. $|\mu|=\sum_{k=1}^r |\mu^{(k)}|$) 
the size of $\mu^{(k)}$ (resp. the size of $\mu$), 
and call an element of $\vL_{n,r}(\Bm)$ an $r$-composition of size $n$. 
We define the map 
$\zeta : \vL_{n,r}(\Bm) \ra \ZZ_{\geq 0}^r$ 
by 
$\zeta(\mu) = (|\mu^{(1)}|,|\mu^{(2)}|, \cdots, |\mu^{(r)}|)$ 
for 
$\mu \in \vL_{n,r}(\Bm)$. 
We also define the partial order \lq\lq $\succeq$" on $\ZZ_{\geq 0}^r$ 
by 
$(a_1,\cdots,a_r) \succ (a_1', \cdots, a'_r)$ 
if 
$\sum_{j=1}^k a_j \geq \sum_{j=1}^k a'_j$ 
for any $k=1,\cdots,r$.
Put 
$\vL_{n,r}^+ (\Bm) = \{  \la \in \vL_{n,r}(\Bm)  \,|\, \la_1^{(k)} \geq \la_2^{(k)} \geq \cdots \geq \la_{m_k}^{(k)} 
	\text{ for any } k=1,\cdots,r \}$. 
We also denote by 
$\vL_{n,r}^+$ 
the set of $r$-partitions of size $n$. 
Then 
we have 
$\vL_{n,r}^+(\Bm)= \vL_{n,r}^+$ 
when $m_k \geq n$ for any $k=1, \cdots, r$.

\para 
\label{Definition M^mu}
For $i=1,\cdots,n$, 
put 
$L_1=T_0$ 
and  
$L_i=T_{i-1} L_{i-1} T_{i-1}$. 
For $\mu \in \vL_{n,r}(\Bm)$, 
put 
\[
m_\mu = \Big( \sum_{w \in \FS_\mu} q ^{\ell (w)} T_w \Big) \Big( \prod_{k=1}^r \prod_{i=1}^{a_k}(L_i - Q_k) \Big), 
\quad 
M^\mu = m_\mu \cdot \,_R \He_{n,r}, 
\] 
where 
$\FS_\mu$ is the Young subgroup of $\FS_n$ with respect to $\mu$, 
and 
$a_k=\sum_{j=1}^{k-1}|\mu^{(j)}|$ 
with 
$a_1=0$.  
The cyclotomic $q$-Schur algebra 
$_R \Sc_{n,r}$ 
associated to 
$_R \He_{n,r}$ 
is defined by 
\[ _R \Sc_{n,r} = \,_R \Sc_{n,r} (\vL_{n,r}(\Bm)) = \End_{_R \He_{n,r}} \Big( \bigoplus_{\mu \in \vL_{n,r}(\Bm) } M^\mu \Big).\]

Put 
$\vG(\Bm) = \{(i,k) \,|\, 1 \leq i \leq m_k, \, 1 \leq k \leq r\}$. 
For $\mu \in \vL_{n,r}(\Bm)$ and $(i,k) \in \vG(\Bm)$, 
we define 
$\s_{(i,k)}^\mu \in \, _R \Sc_{n,r}$ by 
\[ 
\s_{(i,k)}^\mu ( m_\nu \cdot h) = \d_{\mu,\nu} \big( m_\mu (L_{N+1} + L_{N+2} + \cdots + L_{N+\mu_i^{(k)}}) \big) \cdot h 
\quad 
(\nu \in \vL_{n,r}(\Bm), \, h \in \, _R \He_{n,r}), 
\] 
where 
$N=\sum_{l=1}^{k-1} | \mu^{(l)}| + \sum_{j=1}^{i-1} \mu_j^{(k)}$, 
and 
we set $\s_{(i,k)}^\mu =0$ if $\mu_i^{(k)}=0$. 
For $(i,k) \in \vG(\Bm)$, 
put 
$\s_{(i,k)} = \sum_{\mu \in \vL_{n,r}(\Bm)} \s_{(i,k)}^\mu$, 
then 
$\s_{(i,k)} $ 
is a Jucys-Murphy element 
of $_R \Sc_{n,r}$ 
(See \cite{Mat08} for properties of Jucys-Murphy elements). 

\para 
Let  
$\CA = \ZZ [ q,q^{-1}, Q_1,\cdots, Q_r]$, 
where 
$q,Q_1,\cdots,Q_r$ are indeterminate  over $\ZZ$, 
and 
$\CK = \QQ (q, Q_1,\cdots,Q_r)$ 
be the quotient field of $\CA$. 
In order to describe  presentations of $_\CK \Sc_{n,r}$ (resp. $_\CA \Sc_{n,r}$), 
we prepare some notation. 
Put 
$m= \sum_{k=1}^r m_k$. 
Let 
$P=\bigoplus_{i=1}^m \ZZ \ve_i$ 
be the weight lattice of $\Fgl_m$, 
and 
$P^{\vee}=\bigoplus_{i=1}^m \ZZ h_i$  
be the dual weight lattice 
with the natural pairing 
$\lan \, , \, \ran : P \times P^{\vee} \ra \ZZ$ 
such that  
$\lan \ve_i, h_j \ran = \d_{ij}$.
Set $\a_i=\ve_i - \ve_{i+1}$ for $i=1,\cdots,m-1$, 
then 
$\Pi=\{\a_i\,|\, 1\leq i \leq m-1\}$ 
is the set of simple roots, 
and 
$Q=\bigoplus_{i=1}^{m-1} \ZZ\, \a_i$ 
is the root lattice of $\Fgl_m$. 
Put 
$Q^+ = \bigoplus_{i=1}^{m-1} \ZZ_{\geq 0}\, \a_i$. 
We define a partial order 
\lq\lq \,$ \geq$ "
on $P$, 
so called dominance order, 
by 
$\la \geq \mu $ if $\la - \mu \in Q^+$.

We  identify the set 
$\vG (\Bm) $  
with the set  
$\{1,\cdots, m\}$
by the bijection $\g : \vG(\Bm) \ra \{1,\cdots, m\}$ 
given by  
$\g((i,k)) =\sum_{j=1}^{k-1} m_j + i$. 
Put 
$\vG'(\Bm) = \vG \setminus \{(m_r,r)\}$. 
Under this identification, 
we have 
$P=\bigoplus_{i=1}^m \ZZ \ve_i = \bigoplus_{(i,k) \in \vG(\Bm)} \ZZ \ve_{(i,k)}$ 
and 
$Q=\bigoplus_{i=1}^{m-1} \ZZ\, \a_i = \bigoplus_{(i,k) \in \vG'(\Bm)} \ZZ \, \a_{(i,k)}$.  
Then 
we  regard 
$\vL_{n,r}(\Bm)$ 
as a subset of $P$ 
by the injective map 
$\la \mapsto \sum_{(i,k) \in \vG(\Bm)} \la_i^{(k)} \ve_{(i,k)}$. 
For convenience, 
we consider  $(m_k +1,k) = (1, k+1)$ for $(m_k,k) \in \vG'(\Bm)$ 
(resp. $(1-1,k)=(m_{k-1}, k-1)$ for  $(1, k) \in \vG(\Bm) \setminus \{(1,1)\}$).    

Now we have the following  two  presentations of cyclotomic $q$-Schur algebras. 

\begin{thm}[{\cite[Theorem 7.16]{W}}]\
\label{Theorem presentation}
Assume that $m_k \geq n$ for any $k=1,\cdots, r$, 
we have the following presentations of $_\CK \Sc_{n,r}$ and $_\CA \Sc_{n,r}$. 

$(\mathrm{\roi})$    
$_\CK \Sc_{n,r}$ is isomorphic to the algebra over $\CK$ defined by the generators 
$e_{(i,k)}, f_{(i,k)} \\ \big( (i,k) \in \vG'(\Bm) \big)$ and $K_{(i,k)}^{\pm}$ $\big( (i,k) \in \vG(\Bm) \big)$   
with the following defining relations : 
\begin{align}
& K_{(i,k)} K_{(j,l)} = K_{(j,l)} K_{(i,k)}, 
	\quad
	K_{(i,k)} K_{(i,k)}^- = K_{(i,k)}^- K_{(i,k)} =1 
\\
& 	K_{(i,k)} e_{(j,l)} K_{(i,k)}^- = q^{\lan \a_{(j,l)} , h_{(i,k)} \ran} e_{(j,l)}, 
\\
& 	K_{(i,k)} f_{(j,l)} K_{(i,k)}^- = q^{- \lan \a_{(j,l)} , h_{(i,k)} \ran} f_{(j,l)}, 
\\
& e_{(i,k)} f_{(j,l)} - f_{(j,l)} e_{(i,k)} = \d_{(i,k), (j,l)} \eta_{(i,k)}, 
\label{eta relation}
\\
\notag 
& 
\hspace{-3.5em}
\text{where } 
\eta_{(i,k)} = 
	\begin{cases} 
		\dis \frac{K_{(i,k)} K_{(i+1,k)}^- - K_{(i,k)}^- K_{(i+1,k)} }{q-q^{-1}} 
			& \text{if } i \not=m_k, 
		\\[5mm]
		\dis - Q_{k+1} \frac{K_{(m_k ,k)} K_{(1,k+1)}^- - K_{(m_k, k)}^- K_{(1,k+1)} }{q-q^{-1}} 
		\\ \quad 
		+ K_{(m_k,k )} K_{(1, k+1)}^- (q^{-1} g_{(m_k,k)}(f,e) - q g_{(1,k+1)}(f,e) )
			& \text{if } i= m_k,
	\end{cases}
\\
& e_{(i \pm 1, k)} e_{(i,k)}^2 - (q + q^{-1}) e_{(i,k)} e_{(i \pm 1, k)} e_{(i,k)} + e_{(i,k)}^2 e_{(i \pm 1, k)} =0, 
\\
& \notag 
	e_{(i,k)} e_{(j,l)} = e_{(j,l)} e_{(i,k)} \quad ( | \g((i,k)) - \g((j,l))| \geq 2), 
\\
& f_{(i \pm 1, k)} f_{(i,k)}^2 - (q + q^{-1}) f_{(i,k)} f_{(i \pm 1, k)} f_{(i,k)} + f_{(i,k)}^2 f_{(i \pm 1, k)} =0, 
\\
& \notag 
	f_{(i,k)} f_{(j,l)} = f_{(j,l)} f_{(i,k)} \quad ( | \g((i,k)) - \g((j,l))| \geq 2), 
\\
& 
\prod_{(i,k) \in \vG(\Bm)} K_{(i,k)} = q^n,
\\
&
(K_{(i,k)} -1 )(K_{(i,k)}-q) (K_{(i,k)} - q^2)\cdots (K_{(i,k)} - q^n)=0. 
\end{align}

The elements 
$g_{(m_k,k)}(f,e)$, 
$g_{(1,k+1)}(f,e)$ 
in  \eqref{eta relation}  
coincide with the Jucys-Murphy elements $\s_{(m_k,k)}$, $\s_{(1,k+1)}$ respectively,  
which 
are described by generators 
$e_{(i,k)}, f_{(i,k)}$ $\big( (i,k) \in \vG'(\Bm) \big)$ and $K_{(i,k)}^{\pm}$ $\big( (i,k) \in \vG(\Bm) \big)$ 
(see \cite[7.11]{W}). 

Moreover, 
$_\CA \Sc_{n,r}$ is isomorphic to the $\CA$-subalgebra of $_\CK \Sc_{n,r}$ 
generated by 
$e_{(i,k)}^l/[l]!$, $f_{(i,k)}^l/[l]!$ $\big( (i,k) \in \vG'(\Bm), l \geq1 \big)$, 
$K_{(i,k)}^\pm$, 
$\dis \left[ \begin{matrix} K_{(i,k)} ; 0 \\ t \end{matrix} \right] = \prod_{s=1}^t \frac{K_{(i,k)} q^{-s +1} - K_{(i,k)}^{-1} q^{s-1}}{q^s - q^{-s}}$  
$\big( (i,k) \in \vG(\Bm), t \geq 1 \big)$, 
where 
$[l] = \frac{q^l -q^{-l}}{q-q^{-1}}$ 
and   
$[l]!=[l][l-1] \cdots [1]$.
\\

$(\mathrm{\roii})$ 
$_\CK \Sc_{n,r}$ is isomorphic to the algebra over $\CK$ defined by the generators 
$E_{(i,k)}$, $F_{(i,k)}$ $((i,k) \in \vG'(\Bm))$, $1_\la$ $(\la \in \vL_{n,r}(\Bm))$ 
with the following defining relations:  
\begin{align}
&1_\la 1_\mu = \d_{\la,\mu} 1_\la, \quad \sum_{\la \in \vL_{n,r}(\Bm)} 1_\la =1, \label{S-1}\\
&E_{(i,k)} 1_\la = 
	\begin{cases} 
		1_{\la + \a_{(i,k)}} E_{(i,k)} & \text{ if }\la + \a_{(i,k)} \in \vL_{n,r}(\Bm), \label{S-2}\\
		0 & \text{otherwise},
	\end{cases}	\\
&F_{(i,k)} 1_\la = 
	\begin{cases} 
		1_{\la - \a_{(i,k)}} F_{(i,k)} & \text{ if }\la - \a_{(i,k)} \in \vL_{n,r}(\Bm), \label{S-3}\\
		0 & \text{otherwise},
	\end{cases}	\\
&1_{\la} E_{(i,k)}  = 
	\begin{cases} 
		E_{(i,k)} 1_{\la - \a_{(i,k)}} & \text{ if }\la - \a_{(i,k)} \in \vL_{n,r}(\Bm), \label{S-4}\\
		0 & \text{otherwise},
	\end{cases}	\\
&1_\la F_{(i,k)} = 
	\begin{cases} 
		F_{(i,k)} 1_{\la + \a_{(i,k)}} & \text{ if }\la + \a_{(i,k)} \in \vL_{n,r}(\Bm), \label{S-5}\\
		0 & \text{otherwise},
	\end{cases}	\\
&
\label{S-6}
E_{(i,k)}F_{(j,l)} - F_{(j,l)}E_{(i,k)} 
		=\d_{(i,k),(j,l)} \sum_{\la \in \vL_{n,r}} \eta_{(i,k)}^\la, 
\\
\notag
&
\hspace{-3.5em}
\text{where }
\eta_{(i,k)}^\la = 
\begin{cases} 
	[\la_i^{(k)} - \la_{i+1}^{(k)}] 1_\la 
		&\text{if } i\not= m_k, 
	\\[3mm]
	\Big( -Q_{k+1} [ \la_{m_k}^{(k)} - \la_1^{(k+1)}] 
	\\ \quad 
		+ q^{\la_{m_k}^{(k)} - \la_1^{(k+1)}} 
			(q^{-1} g_{(m_k,k)}^\la(F,E) - q g_{(1,k+1)}^\la(F,E) ) \Big) 1_\la 
		&\text{if } i=m_k,
\end{cases}
\\
&E_{(i \pm 1,k)}(E_{(i,k)})^2 - (q+q^{-1}) E_{(i,k)} E_{(i \pm 1,k)} E_{(i,k)} 
	+ (E_{(i,k)})^2 E_{(i \pm 1,k)}=0 , \label{S-7}\\
& E_{(i,k)} E_{(j,l)}= E_{(j,l)} E_{(i,k)} \qquad (|\g((i,k))- \g((j,l))| \geq 2), \notag \\
&F_{(i \pm 1,k)}(F_{(i,k)})^2 - (q+q^{-1}) F_{(i,k)} F_{(i \pm 1,k)} F_{(i,k)} 
	+ (F_{(i,k)})^2 F_{(i \pm 1,k)}=0,  \label{S-8}\\
& F_{(i,k)} F_{(j,l)}= F_{(j,l)} F_{(i,k)} \qquad  (|\g((i,k))- \g((j,l))| \geq 2), \notag 
\end{align}
The elements 
$g^\la_{(m_k,k)}(F,E)$, 
$g^\la_{(1,k+1)}(F,E)$ 
in  \eqref{S-6}
coincide with $\s_{(m_k,k)}^\la$, $\s_{(1,k+1)}^\la$ respectively, 
which  
are described by generators 
$E_{(i,k)}, F_{(i,k)}$ $\big( (i,k) \in \vG'(\Bm) \big)$ 
(see \cite[7.1-7.4]{W}). 

Moreover, 
$_\CA \Sc_{n,r}$ is isomorphic to the $\CA$-subalgebra of $_\CK \Sc_{n,r}$ 
generated by 
$E_{(i,k)}^l/[l]!$, 
$F_{(i,k)}^l/[l]!$ 
$\big( (i,k) \in \vG'(\Bm),  \, l \geq 1\big)$, 
$1_\la$ $(\la \in \vL_{n,r}(\Bm) )$. 
\end{thm}

\remark 
\label{Remark weight cut}
In \cite{W}, 
we treated only the case where $m_k =n$ for any $k=1,\cdots,r$. 
We obtain Theorem \ref{Theorem presentation} for the general case in the same way under the condition $m_k \geq n$ for any $k=1,\cdots, r$. 
However, in the case where $m_k < n$ for some $k$, 
we do not have the presentation of $\Sc_{n,r}(\vL_{n,r}(\Bm))$ as in the above theorem. 
In such a case, 
we have the following realization of $\Sc_{n,r}(\vL_{n,r}(\Bm))$. 
First, we take $\wt{\Bm}= (\wt{m}_1,\cdots, \wt{m}_r) \in \ZZ_{>0}^r$ 
such that 
$\wt{m}_k \geq n$ and $\wt{m}_k \geq m_k$ for any $k=1,\cdots, r$. 
Then, we can regard $\vL_{n,r}(\Bm)$ as a subset of $\vL_{n,r}(\wt{\Bm})$ in the natural way. 
We have the presentation of $\Sc_{n,r}(\vL_{n,r}(\wt{\Bm}))$ by the theorem, 
and 
we have 
$\Sc_{n,r}(\vL_{n,r}(\Bm)) = 1_{\Bm} \Sc_{n,r}(\vL_{n,r}(\wt{\Bm})) 1_{\Bm}$, 
where $1_{\Bm} = \sum_{\la \in \vL_{n,r}(\Bm)} 1_\la \in \Sc_{n,r}(\vL_{n,r}(\wt{\Bm}))$.

\para 
\label{Definition Weyl}
\textit{Weyl modules} (see \cite{W} and \cite{DJM98} for more details).
Let 
$_\CA \Sc_{n,r}^+$ (resp. $_\CA \Sc_{n,r}^-$) 
be the subalgebra of $_\CA \Sc_{n,r}$ 
generated by 
$E_{(i,k)}^{l}/[l]!$ 
(resp. $F_{(i,k)}^{l}/[l]!$) 
for $(i,k) \in \vG'(\Bm)$ and $l \geq 1$. 
Let  
$_\CA \Sc_{n,r}^0$ 
be the subalgebra of $_\CA \Sc_{n,r}$ 
generated by 
$1_\la$ for $\la \in \vL_{n,r}(\Bm)$. 
Then 
$_\CA \Sc_{n,r}$ has the triangular decomposition 
$_\CA \Sc_{n,r} = \, _\CA \Sc_{n,r}^- \, _\CA \Sc_{n,r}^0 \, _\CA \Sc_{n,r}^+$. 
We denote by 
$_\CA \Sc_{n,r}^{\geq 0}$ 
the subalgebra of $_\CA \Sc_{n,r}$ 
generated by 
$_\CA \Sc_{n,r}^+$ and $_\CA \Sc_{n,r}^0$. 

Let $R$ be an  arbitrary commutative ring, 
and we take parameters 
$\wt{q}, \wt{Q}_1,\cdots,\wt{Q}_r \in R$ 
such that 
$\wt{q}$ is invertible in $R$. 
We consider 
the specialized cyclotomic $q$-Schur algebra 
$_R \Sc_{n,r} = R \otimes_{\CA}\,_\CA \Sc_{n,r}$ 
through the ring homomorphism 
$\CA \ra R$, 
$q \mapsto \wt{q}$, $Q_k \mapsto \wt{Q}_k$ ($1 \leq k \leq r$). 
Then $_R \Sc_{n,r}$ also has the triangular decomposition 
$_R \Sc_{n,r} = \, _R \Sc_{n,r}^- \, _R \Sc_{n,r}^0 \, _R \Sc_{n,r}^+$ 
which comes from the triangular decomposition of $_\CA \Sc_{n,r}$.

For $\la \in \vL_{n,r}^+(\Bm)$, 
we define the 1-dimensional $_R \Sc_{n,r}^{\geq 0}$-module 
$\theta_\la = R v_\la$ 
by 
$E_{(i,k)} \cdot v_\la =0$ $\big( (i,k) \in \vG'(\Bm) \big)$ 
and 
$1_\mu \cdot v_\la = \d_{\la,\mu} v_\la $ $(\mu \in \vL_{n,r}(\Bm))$. 
Then the  Weyl module $_R W(\la)$ of $_R \Sc_{n,r}$ 
is defined as the induced module $_R \Sc_{n,r} \otimes_{_R \Sc_{n,r}^{\geq 0}} \theta_\la$ of $\theta_\la$ 
for $\la \in \vL_{n,r}^+(\Bm)$.

When $R$ is a field, 
it is known that 
$_R W(\la)$ has the unique simple top $_R L(\la)$, 
and that 
$\{\,_R L(\la) \,|\, \la \in \vL_{n,r}^+(\Bm)\}$ 
gives a complete set of non-isomorphic (left) simple $_R \Sc_{n,r}$-modules. 
Moreover, 
it is known that 
$_\CK \Sc_{n,r}$ 
is semi-simple,  
and that 
$\{_\CK W(\la) \,|\, \la \in \vL_{n,r}^+(\Bm)\}$ 
gives a complete set of non-isomorphic (left) simple $_\CK \Sc_{n,r}$-modules.

\para 
By \eqref{S-1}, 
the identity element $1$ of $_R \Sc_{n,r}$ 
decomposes to a sum of  pairwise orthogonal idempotents 
indexed by $\vL_{n,r}(\Bm)$,  
namely we have  $1 = \sum_{\la \in \vL_{n,r}(\Bm)} 1_\la $. 
Thanks to this decomposition, 
for $_R \Sc_{n,r}$-module $M$, 
we have the decomposition 
$M = \bigoplus_{\la \in \vL_{n,r}(\Bm)} 1_\la M$ 
as $R$-modules.  
By the isomorphism between the first presentation and the second presentation of $\Sc_{n,r}$ in Theorem \ref{Theorem presentation} 
(see \cite[Proposition 7.12]{W} for this isomorphism),   
we see that 
$K_{(i,k)}$ acts on $1_\la M$ 
by multiplying the scalar $q^{\la_i^{(k)}}$, 
namely we have 
$1_\la M = \{ m \in M \,|\, K_{(i,k)} \cdot m = q^{\la_i^{(k)}} m \text{ for } (i,k) \in \vG(\Bm)\}$.  
We call 
$1_\la M$ the weight space of weight $\la$ 
(or $\la$-weight space simply),  
and denote  by $M_\la$. 



\section{$U_q(\Fg)$-crystal structure on Weyl module $W(\la)$ of $\Sc_{n,r}$} 
\label{section crystal}
In the section \ref{section crystal} -- section \ref{section character}, 
we consider only the cyclotomic $q$-Schur algebra $_\CK \Sc_{n,r}$ over $\CK$. 
Hence, we omit the subscript $\CK$.  
Moreover, 
we assume that 
$m_k \geq n$ for any $k=1,\cdots,r$ in this section.

\para 
\label{Definition Phi_Fg}
Let $\Fg = \Fgl_{m_1} \oplus \cdots \oplus \Fgl_{m_r}$ 
be the Levi subalgebra of $\Fgl_m$, 
and 
$U_q(\Fg) \cong U_q(\Fgl_{m_1}) \otimes \cdots \otimes U_q(\Fgl_{m_r})$ 
be the quantum group over $\CK$ corresponding to $\Fg$. 
Put $\vG'_{\Fg}(\Bm) = \vG(\Bm) \setminus \{(m_k,k) \,|\, 1 \leq k \leq r\}$. 
Let 
$e_{(i,k)}, f_{(i,k)}$ ($(i,k) \in \vG'_{\Fg}(\Bm)$), 
$K_{(i,k)}^{\pm}$ ($(i,k) \in \vG(\Bm)$) 
be the generators of $U_q(\Fg)$, 
where 
$e_{(i,k)}, f_{(i,k)}, K_{(j,k)}^\pm$ ($1 \leq i \leq m_k-1$, $1 \leq j \leq m_k$) 
is the usual Chevalley  generators of $U_q(\Fgl_{m_k})$.

By the presentation of $\Sc_{n,r}$ (Theorem \ref{Theorem presentation}), 
we can define the algebra homomorphism 
$\Phi_\Fg : U_q(\Fg) \ra \Sc_{n,r}$ 
sending generators of $U_q(\Fg)$ to the corresponding generators of $\Sc_{n,r}$ denoted by the same symbol. 
Note that 
$\Phi_\Fg$ is not surjective without the case where $r=1$. 
We have the following lemma which describes the image of $\Phi_\Fg$.

\begin{lem}\
\label{Lemma image Phi_g}
\begin{enumerate}
\item 
$\dis \Phi_\Fg(U_q(\Fg)) \cong \bigoplus_{\eta = (n_1,\cdots,n_r) \atop n_1+ \cdots + n_r=n} 
	\Sc_{n_1,1}^\eta(\vL_{n_1,1}(m_1)) \otimes \cdots \otimes \Sc_{n_r,1}^\eta(\vL_{n_r,1}(m_r))$, 
\\
where $\Sc_{n_k,1}^\eta(\vL_{n_k,1}(m_k))$ is the $q$-Schur algebra associated to the symmetric group $\FS_{n_k}$ of degree $n_k$.  

\item 
Let 
$_\CA U_q(\Fg)$  be the $\CA$-form of $U_q(\Fg)$ by taking the divided powers. 
Then we have 
\[ 
	\Phi_\Fg(\,_\CA U_q(\Fg)) \cong \bigoplus_{\eta = (n_1,\cdots,n_r) \atop n_1+ \cdots + n_r=n} 
	\, _\CA \Sc_{n_1,1}^\eta(\vL_{n_1,1}(m_1)) \otimes \cdots \otimes \, _\CA \Sc_{n_r,1}^\eta(\vL_{n_r,1}(m_r))
\]
\end{enumerate}
\end{lem}

\begin{proof}
Let 
$e_i^{\eta_k}$, $f_i^{\eta_k}$ $(1 \leq i \leq m_k -1)$, 
$K_i^{\eta_k \pm}$ ($1 \leq i \leq m_k$) 
be the generators of $\Sc_{n_k,1}^{\eta} ( \vL_{n_k,1} (m_k))$ 
in  Theorem \ref{Theorem presentation} (\roi). 
Then, 
we define the homomorphism of algebras 
\[\vf : U_q(\Fg) \ra \bigoplus_{\eta = (n_1,\cdots,n_r) \atop n_1+ \cdots + n_r=n} 
	\Sc_{n_1,1}^\eta(\vL_{n_1,1}(m_1)) \otimes \cdots \otimes \Sc_{n_r,1}^\eta(\vL_{n_r,1}(m_r))
\] 
by 
\begin{align*}
&
\vf(e_{(i,k)}) = \sum_{\eta =(n_1, \cdots, n_r) \atop n_1+ \cdots +n_r=n} 
	\underbrace{1 \otimes \cdots \otimes 1}_{k-1} \otimes e_i^{\eta_k} \otimes 1  \otimes \cdots \otimes 1, 
\\
&
\vf(f_{(i,k)}) = \sum_{\eta =(n_1, \cdots, n_r) \atop n_1+ \cdots +n_r=n} 
	\underbrace{1 \otimes \cdots \otimes 1}_{k-1} \otimes f_i^{\eta_k} \otimes 1  \otimes \cdots \otimes 1, 
\\
&
\vf(K_{(i,k)}^\pm) = \sum_{\eta =(n_1, \cdots, n_r) \atop n_1+ \cdots +n_r=n} 
	\underbrace{1 \otimes \cdots \otimes 1}_{k-1} \otimes K_i^{\eta_k \pm} \otimes 1  \otimes \cdots \otimes 1
\end{align*}
for generators 
$e_{(i,k)}$, $f_{(i,k)}$ ($(i,k) \in \vG'_{\Fg}(\Bm)$), $K_{(i,k)}^\pm$ ($(i,k) \in \vG(\Bm)$) 
of $U_q(\Fg)$. 
(We can easily check that $\vf$ is well-defined by Theorem \ref{Theorem presentation} (\roi).)

We also define  the homomorphism of algebras  
\[
 \psi : \bigoplus_{\eta = (n_1,\cdots,n_r) \atop n_1+ \cdots + n_r=n} 
	\Sc_{n_1,1}^\eta(\vL_{n_1,1}(m_1)) \otimes \cdots \otimes \Sc_{n_r,1}^\eta(\vL_{n_r,1}(m_r)) 
	\ra 
	\Phi_\Fg(U_q(\Fg))
\]
by 
\begin{align*}
&
\psi(\underbrace{1 \otimes \cdots \otimes 1}_{k-1} \otimes e_i^{\eta_k} \otimes 1  \otimes \cdots \otimes 1) 
=(\sum_{\mu \in \vL_{n,r}(\Bm) \atop \zeta(\mu)=\eta} 1_\mu) \cdot 
		\Phi(e_{(i,k)}) \cdot (\sum_{\mu \in \vL_{n,r}(\Bm) \atop \zeta(\mu)=\eta} 1_\mu), 
\\
&
\psi(\underbrace{1 \otimes \cdots \otimes 1}_{k-1} \otimes f_i^{\eta_k} \otimes 1  \otimes \cdots \otimes 1) 
=(\sum_{\mu \in \vL_{n,r}(\Bm) \atop \zeta(\mu)=\eta} 1_\mu) \cdot 
		\Phi(f_{(i,k)}) \cdot (\sum_{\mu \in \vL_{n,r}(\Bm) \atop \zeta(\mu)=\eta} 1_\mu),
\\
&
\psi(\underbrace{1 \otimes \cdots \otimes 1}_{k-1} \otimes K_i^{\eta_k \pm } \otimes 1  \otimes \cdots \otimes 1) 
=(\sum_{\mu \in \vL_{n,r}(\Bm) \atop \zeta(\mu)=\eta} 1_\mu) \cdot  
		\Phi(K_{(i,k)}) \cdot (\sum_{\mu \in \vL_{n,r}(\Bm) \atop \zeta(\mu)=\eta} 1_\mu)
\end{align*}
for each generators of  
$\bigoplus_{\eta = (n_1,\cdots,n_r) \atop n_1+ \cdots + n_r=n} 
	\Sc_{n_1,1}^\eta(\vL_{n_1,1}(m_1)) \otimes \cdots \otimes \Sc_{n_r,1}^\eta(\vL_{n_r,1}(m_r))$. 
(We can check the well-definedness by direct calculations.)
It is clear that 
$\psi \circ \vf = \Phi_{\Fg}$. 
Thus,  
$\psi$ is surjective. 
Moreover, 
by comparing the  simple modules appearing 
in 
$\bigoplus_{\eta = (n_1,\cdots,n_r) \atop n_1+ \cdots + n_r=n} 
	\Sc_{n_1,1}^\eta(\vL_{n_1,1}(m_1)) \otimes \cdots \otimes \Sc_{n_r,1}^\eta(\vL_{n_r,1}(m_r))$
and 
in $\Phi_\Fg(U_q(\Fg))$ 
as $U_q(\Fg)$-modules 
through $\vf$ and $\Phi_{\Fg}$ respectively, 
we see that 
$\psi$ is an isomorphism. 
(Note that both $\bigoplus_{\eta = (n_1,\cdots,n_r) \atop n_1+ \cdots + n_r=n} 
	\Sc_{n_1,1}^\eta(\vL_{n_1,1}(m_1)) \otimes \cdots \otimes \Sc_{n_r,1}^\eta(\vL_{n_r,1}(m_r))$ 
	and $U_q(\Fg)$ 
	are semi-simple.) 
(\roii)  
follows from (\roi) 
by restricting $\Phi_{\Fg}$ to $_\CA U_q(\Fg)$. 
\end{proof}

\para 
For an $\Sc_{n,r}$-module $M$, 
we regard $M$ as a $U_q(\Fg)$-module through the homomorphism $\Phi_\Fg$. 
Then,  
by Lemma \ref{Lemma image Phi_g} (or by investigating  weights directly), 
we see that 
a simple $U_q(\Fg)$-module appearing in $M$ as a composition factor 
is the form $W(\la^{(1)}) \boxtimes \cdots \boxtimes W(\la^{(r)})$ for some $\la \in \vL_{n,r}^+(\Bm)$, 
where 
$W(\la^{(k)})$ is the highest weight $U_q(\Fgl_{m_k})$-module of highest weight $\la^{(k)}$. 
Hence, the Weyl module $W(\la)$ of $\Sc_{n,r}$ decomposes as follows: 
\begin{align} 
\label{def b_la_mu}
W(\la) \cong   \bigoplus_{\mu \in \vL_{n,r}^+(\Bm)} \Big(W(\mu^{(1)}) \boxtimes \cdots \boxtimes W(\mu^{(r)})\Big)^{\oplus \b_{\la\mu}} 
\quad \text{ as $U_q(\Fg)$-modules}.
\end{align}

\para 
In order to compute the multiplicity $\b_{\la\mu}$ in \eqref{def b_la_mu},  
we will describe the $U_q(\Fg)$-crystal structure on $W(\la)$. 
For such a purpose, 
we prepare some notation of combinatorics. 

For $\mu \in \vL_{n,r}(\Bm)$, 
the diagram $[\mu]$ of $\mu$ is 
the set 
$\{(i,j,k) \in \ZZ^3 \,|\, 1 \leq i \leq m_k, \, 1 \leq j \leq \mu_i^{(k)}, \, 1 \leq k \leq r\}$. 
For $\la \in \vL_{n,r}^+(\Bm)$ and $\mu \in \vL_{n,r}(\Bm)$, 
a tableau of shape $\la$ with weight $\mu$ 
is a map 
$T : [\la] \ra \{(a,c) \in \ZZ \times \ZZ \,|\, a \geq 1, 1 \leq c \leq r\}$ 
such that 
$\mu_i^{(k)} = \sharp \{x \in [\la] \,|\, T(x) = (i,k)\}$. 
We define the order on $\ZZ \times \ZZ$ 
by 
$(a,c) \geq (a',c')$ if either $c > c'$, or $c=c'$ and $a \geq a'$.
For a tableau $T$ of shape $\la$ with weight $\mu$, 
we say that $T$ is  semi-standard if $T$ satisfies the following conditions: 
\begin{enumerate}
\item If $T((i,j,k)) =(a,c)$, then $k \leq c$,
\item $T((i,j,k)) \leq T((i,j+1,k))$ if $(i,j+1,k) \in [\la]$, 
\item $T((i,j,k)) < T((i+1,j,k))$ if $(i+1,j,k) \in [\la]$. 
\end{enumerate}

For $\la \in \vL_{n,r}^+(\Bm)$, $\mu \in \vL_{n,r}(\Bm)$, 
we denote by 
$\CT_0(\la,\mu)$ 
the set of semi-standard tableaux of shape $\la$ with weight $\mu$. 
Put $\CT_0(\la) = \bigcup_{\mu \in \vL_{n,r}(\Bm)} \CT_0(\la,\mu)$. 
We identify a semi-standard tableau 
with a Young tableau 
as the following example. 
\\
For $\la =((3,2),(3,1),(1,1))$,  $\mu=((2,1),(2,2),(3,1))$ 
\[
T= 
\left( 
\,\,
\begin{array}{|c|c|c|}
\hline 
(1,1)&(1,1)&(1,2)
\\
\hline
(2,1)&(1,3)
\\
\cline{1-2}
\end{array}
\,,\,
\begin{array}{|c|c|c|}
\hline 
(1,2)&(2,2)&(1,3)
\\
\hline
(2,2)
\\
\cline{1-1}
\end{array}
\,,\,
\begin{array}{|c|}
\hline 
(1,3)
\\
\hline
(2,3)
\\
\cline{1-1}
\end{array}
\,\,
\right) 
\in \CT_0(\la,\mu), 
\]
where 
$T((1,1,1))=(1,1)$, $T((1,2,1))=(1,1)$, $\cdots$, $T(( 2,1,3))=(2,3)$. 

By \cite{DJM98}, 
it is known that 
there exists a bijection 
between 
$\CT_0(\la,\mu)$ and 
a basis of $W(\la)_\mu$. 
Hence, 
we will describe a $U_q(\Fg)$-crystal structure on $\CT_0(\la)$ 
which is isomorphic to the $U_q(\Fg)$-crystal on $W(\la)$.

\para 
By \eqref{def b_la_mu}, 
for $\la \in \vL_{n,r}^+(\Bm)$, $\mu \in \vL_{n,r}(\Bm)$, 
we have 
\begin{align}
\label{number semi-standard}
\sharp \CT_0(\la,\mu) 
&= \dim W(\la)_\mu 
\\
\notag 
&= \sum_{\nu \in \vL_{n,r}^+(\Bm)} \b_{\la\nu} \cdot \dim \Big( W(\nu^{(1)}) \boxtimes \cdots \boxtimes W(\nu^{(r)}) \Big)_\mu 
\\
\notag 
&= \sum_{\nu \in \vL_{n,r}^+(\Bm)} \b_{\la\nu} \prod_{k=1}^r \dim W(\nu^{(k)})_{\mu^{(k)}} 
\\
\notag 
&= \sum_{\nu \in \vL_{n,r}^+(\Bm)} \b_{\la\nu} \prod_{k=1}^r \sharp \CT_0(\nu^{(k)}, \mu^{(k)})  
\\
\notag 
&= \sum_{\nu \in \vL_{n,r}^+(\Bm)} \b_{\la\nu} \prod_{k=1}^r K_{\nu^{(k)} \mu^{(k)}}, 
\end{align}
where 
$K_{\nu^{(k)} \mu^{(k)}}$ is the Kostka number. 
We have the following properties of $\b_{\la\mu}$.

\begin{lem}\ 
\label{Lemma properties b_la_mu}
\begin{enumerate}
\item 
For $\la \in \vL_{n,r}^+(\Bm)$, 
$\b_{\la\la}=1$. 

\item 
For $\la,\mu \in \vL_{n,r}^+(\Bm)$, 
if $\b_{\la\mu} \not=0$, we have $\la \geq \mu$. 

\item 
For $\la,\mu \in \vL_{n,r}^+(\Bm)$,  
if $\la \not= \mu$ and $\zeta (\la) = \zeta (\mu)$, 
we have 
$\b_{\la \mu} =0$. 

\item 
For 
$\la,\mu \in \vL_{n,r}^+(\Bm)$ such that $\zeta(\la) \not= \zeta (\mu)$, 
if 
$\CT_0(\la,\nu) = \emptyset $ 
for any $\nu \in \vL_{n,r}^+(\Bm)$ such that $\zeta( \nu) = \zeta (\mu)$ and $\nu > \mu$, 
then 
we have 
$\b_{\la\mu} = \sharp \CT_0(\la,\mu)$. 
\end{enumerate}
\end{lem}

\begin{proof}
(\roi)
From the definition of $W(\la)$, 
we have 
$W(\la) = \Sc_{n,r}^- \cdot v_\la$, 
where we denote 
$1 \otimes v_\la \in \Sc_{n,r} \otimes_{\Sc_{n,r}^{\geq 0}} \theta_\la $ by $v_\la$ simply. 
Thus, 
we have that 
$W(\la)_\la = \CK v_\la$, 
and that 
$v_\la$ is a highest weight vector of highest weight $\la$ in 
$U_q(\Fg)$-module $W(\la)$. 
This implies that 
$\b_{\la\la}=1$. 

(\roii) 
$\b_{\la \mu} \not=0 \Rightarrow W(\la)_\mu \not=0 \Rightarrow \la \geq \mu$. 

(\roiii) 
Assume that 
$\la \not=\mu$ and $\zeta(\la)=\zeta(\mu)$. 
By \eqref{number semi-standard}, 
we have 
\begin{align}
\label{number semi-standard same}
\sharp \CT_0(\la,\mu) 
= 
&\b_{\la\la} \prod_{k=1}^r \sharp \CT_0(\la^{(k)}, \mu^{(k)}) 
+ \b_{\la\mu} \prod_{k=1}^r \sharp \CT_0(\mu^{(k)}, \mu^{(k)}) 
\\
\notag 
&\quad + \sum_{\nu \in \vL_{n,r}^+(\Bm) \atop \nu \not=\la,\mu} \b_{\la\nu} \prod_{k=1}^r \sharp \CT_0(\nu^{(k)}, \mu^{(k)}). 
\end{align}
This 
implies that 
$\b_{\la\mu}=0$ 
since  
$\sharp \CT_0(\mu^{(k)},\mu^{(k)})=1$,  
and 
$\sharp \CT_0(\la,\mu) = \prod_{k=1}^r  \sharp \CT_0(\la^{(k)}, \mu^{(k)})$ 
if $\zeta(\la)=\zeta(\mu)$.   

(\roiv) 
Note that $\prod_{k=1}^r \sharp \CT_0(\nu^{(k)}, \mu^{(k)}) =0$ if $\zeta(\nu) \not= \zeta (\mu)$ or $\nu \not \geq \mu$, 
and that 
$\prod_{k=1}^r \sharp \CT_0(\nu^{(k)}, \mu^{(k)}) = \CT_0(\nu,\mu)$ 
if $\zeta(\nu)= \zeta(\mu)$.  
Then \eqref{number semi-standard} combining with the assumption of (\roiv) 
implies 
$\sharp \CT_0(\la,\mu) = \b_{\la\mu} \sharp \CT_0(\mu,\mu) = \b_{\la\mu}$ 
since 
$\b_{\la\nu}=0$ if $\CT_0(\la,\nu)=\emptyset$. 
\end{proof}

\para 
\label{def crystal T(la) start}
For $\la \in \vL_{n,r}(\Bm)$, 
we define the total order \lq\lq $ \succeq $" on the diagram $[\la]$ 
by 
$(i,j,k) \succ (i',j',k')$ 
if  $k > k'$,  $k=k'$ and $ j > j'$  or if  $ k=k', j=j' $ and  $i < i'$.   
For an example, we have 
\[ (5,4,2) \succ (2,3,2) \succ (5,3,2) \succ (6,4,1). \]

\para 
\label{Definition equiv. on CT_0(la)}
We define the equivalence relation \lq\lq $\sim $" on $\CT_0(\la)$ 
by 
$T \sim T'$ if 
$\{x \in [\la] \,|\, T(x) = (i,k) \text{ for some } i=1,\cdots, m_k\} = \{ y \in [\la] \,|\, T'(y) = (j,k) \text{ for some } j=1,\cdots, m_k\}$ 
for any $k=1,\cdots, r$. 
By the definition,  
for $T \in \CT_0(\la,\mu)$ and $T' \in \CT_0(\la,\nu)$, 
we have 
\begin{align}
\label{zeta invariant}
\zeta(\mu)= \zeta(\nu) \text{ if } T \sim T'. 
\end{align}

\example 
Put 
\begin{align*} 
&
T_1=
\left( \, 
\begin{array}{|c|c|}
\hline 
(1,1)&(1,1)
\\
\hline
(1,2)&(2,2)
\\
\cline{1-2}
\end{array}
\, , \,
\begin{array}{|c|c|}
\hline 
(1,2)&(2,2)
\\
\hline
(3,2)
\\
\cline{1-1}
\end{array}
\,\right), 
\,\,
T_2=
\left( \, 
\begin{array}{|c|c|}
\hline 
(1,1)&(2,1)
\\
\hline
(1,2)&(3,2)
\\
\cline{1-2}
\end{array}
\, , \,
\begin{array}{|c|c|}
\hline 
(2,2)&(2,2)
\\
\hline
(4,2)
\\
\cline{1-1}
\end{array}
\,\right), 
\\[3mm]
&
T_3=
\left( \, 
\begin{array}{|c|c|}
\hline 
(1,1)&(1,2)
\\
\hline
(2,1)&(3,2)
\\
\cline{1-2}
\end{array}
\, , \,
\begin{array}{|c|c|}
\hline 
(2,2)&(2,2)
\\
\hline
(4,2)
\\
\cline{1-1}
\end{array}
\,\right), 
\,\, 
T_4=
\left( \, 
\begin{array}{|c|c|}
\hline 
(1,1)&(2,2)
\\
\hline
(3,1)&(3,2)
\\
\cline{1-2}
\end{array}
\, , \,
\begin{array}{|c|c|}
\hline 
(1,2)&(1,2)
\\
\hline
(2,2)
\\
\cline{1-1}
\end{array}
\,\right).
\end{align*}
Then, 
we have 
$T_1 \sim T_2$, $T_2 \not\sim T_3$ and $T_3 \sim T_4$.

\para 
\label{Definition reading}
Let $V_{m_k}$ be the vector representation of $U_q(\Fgl_{m_k})$ 
with a natural basis $\{ v_1, v_2, \cdots, v_{m_k}\}$. 
Let $\CA_0$ be the localization of $\QQ(Q_1,\cdots,Q_r) [q]$ at $q=0$. 
Put 
$\CL_{m_k} = \bigoplus_{j=1}^{m_k} \CA_0 \cdot v_j$,  
$\bo{j} = v_j + q \CL_{m_k} \in \CL_{m_k}/ q \CL_{m_k}$ 
and 
$\CB_{m_k} = \big\{ \bo{j} \,|\, 1 \leq j \leq m_k \big\} $. 
Then 
$(\CL_{m_k}, \CB_{m_k})$ gives the crystal basis of $V_{m_k}$.  
We denote by 
$\CB_{m_1}^{\otimes n_1} \boxtimes \cdots \boxtimes \CB_{m_r}^{\otimes n_r}$ 
the $U_q(\Fg)$-crystal corresponding to 
$U_q(\Fg)$-module 
$V_{m_1}^{\otimes n_1} \boxtimes \cdots \boxtimes V_{m_r}^{\otimes n_r}$.

Let 
$\CT_0(\la) = \bigcup_{t} \CT_0(\la)[t]$  
be the decomposition to equivalence classes 
with respect to the equivalence relation \lq\lq $\sim$".  
For an each equivalence class $\CT_0(\la)[t]$, 
put 
$(n_1,\cdots,n_r) = \zeta (\mu)$ 
for some $\mu$ such that $\CT_0(\la,\mu) \cap \CT_0(\la)[t] \not= \emptyset$ 
(note \eqref{zeta invariant}), 
and 
we define the map 
$\Psi_t^\la : \CT_0(\la)[t] \ra \CB_{m_k}^{\otimes n_1} \boxtimes \cdots \boxtimes \CB_{m_k}^{\otimes n_r}$ 
as 
\begin{align*} 
&\Psi_t^\la (T) =\big( \bo{i_1^{(1)}} \otimes \cdots \otimes \bo{i_{n_1}^{(1)}} \big)
	\boxtimes \cdots \boxtimes  \big( \bo{i_1^{(r)}} \otimes \cdots \otimes \bo{i_{n_r}^{(r)}} \big) 
\end{align*}
satisfying  the following three conditions: 
\begin{enumerate} 
\item 
$\{ x \in [\la] \,|\, T(x)= (i,k) \text{ for some } i=1,\cdots, m_k\} = \{ x_1^{(k)}, x_2^{(k)}, \cdots, x_{n_k}^{(k)} \}$ 
\\
for  $k=1,\cdots,r$. 

\item 
$x_1^{(k)} \succ x_{2}^{(k)} \succ \cdots \succ x_{n_k}^{(k)}$ 
for $k=1,\cdots, r$. 

\item 
$T(x_j^{(k)}) = (i_j^{(k)},k)$ for $1 \leq j \leq n_k$, $1 \leq k \leq r$. 
\end{enumerate}
Namely, 
$\big( \bo{i_1^{(k)}} \otimes \cdots \otimes \bo{i_{n_k}^{(k)}} \big)$ 
in $\Psi_t^\la (T)$ 
is obtained by reading 
the first coordinate of $T(x)$ 
for 
$x \in  [\la]$ 
such that 
$T(x)=(i,k)$ for some $i=1,\cdots,m_k$ 
in the order \lq\lq $ \succeq$" on $[\la]$.

\example 
For 
\[
T= 
\left( 
\,\,
\begin{array}{|c|c|c|}
\hline 
(1,1)&(1,1)&(1,2)
\\
\hline
(2,1)&(1,3)
\\
\cline{1-2}
\end{array}
\,,\,
\begin{array}{|c|c|c|}
\hline 
(1,2)&(2,2)&(1,3)
\\
\hline
(2,2)
\\
\cline{1-1}
\end{array}
\,,\,
\begin{array}{|c|}
\hline 
(1,3)
\\
\hline
(2,3)
\\
\cline{1-1}
\end{array}
\,\,
\right) 
\in \CT_0(\la)[t], 
\]
we have 
$\Psi_t^\la(T)=
	\big( \bo{1} \otimes \bo{1} \otimes \bo{2} \big) 
	\boxtimes 
	\big( \bo{2} \otimes \bo{1} \otimes \bo{2} \otimes \bo{1} \big)
	\boxtimes 
	\big(\bo{1} \otimes \bo{2} \otimes \bo{1} \otimes \bo{1} \big)
$.

\remark 
In the case where $r=1$, 
$\CT_0(\la)$ has only one equivalence class (itself) with respect to \lq\lq $\sim$", 
and 
$\Psi^\la$ coincides with the Far-Eastern reading given in \cite[\S 3]{KN} (see also \cite[Ch. 7]{HK}).

\para 
Let 
$\wt{e}_{(i,k)}$, $\wt{f}_{(i,k)}$ ($(i,k) \in \vG'_\Fg(\Bm)$) 
be the Kashiwara operators on $U_q(\Fg)$-crystal $ \CB_{m_k}^{\otimes n_1} \boxtimes \cdots \boxtimes \CB_{m_k}^{\otimes n_r}$. 
Then we have the following proposition.

\begin{prop}\
\label{Proposition reading}
For an each equivalence class $\CT_0(\la)[t]$ of $\CT_0(\la)$,  
we have the followings. 
\begin{enumerate}
\item 
The map $\Psi_t^\la : \CT_0(\la)[t] \ra \CB_{m_k}^{\otimes n_1} \boxtimes \cdots \boxtimes \CB_{m_k}^{\otimes n_r}$ 
is injective. 

\item   
$\Psi_t^\la \big( \CT_0(\la)[t] \big) \cup \{0\}$ 
is stable under the Kashiwara operators $\wt{e}_{(i,k)}$, $\wt{f}_{(i,k)}$ $\big( (i,k) \in \vG'_\Fg(\Bm) \big)$ .
\end{enumerate}
\end{prop}
\begin{proof} 
(\roi) is clear from the definitions. 
(\roii) is proven in a similar way as in the case of type $A$ ($r=1$)
(see \cite{KN} or \cite[Theorem 7.3.6]{HK}). 
\end{proof}

\para 
\label{def crystal T(la) end}
By Proposition \ref{Proposition reading}, 
we  define the $U_q(\Fg)$-crystal structure on $\CT_0(\la)[t]$ through $\Psi_t^\la$, 
and 
also define  the $U_q(\Fg)$-crystal structure on $\CT_0(\la)$.  
Note that 
the $U_q(\Fg)$-crystal graphs of $\CT_0(\la)[t]$ 
and of   $\CT_0(\la)[t']$ 
are disconnected 
in the $U_q(\Fg)$-crystal graph of $\CT_0(\la)$ 
if 
$\CT_0(\la)[t]$ 
is a different equivalence class  
from $\CT_0(\la)[t']$. 
For $T \in \CT_0(\la)$, 
we say that 
$T$ is $U_q(\Fg)$-singular 
if 
$\wt{e}_{(i,k)} \cdot T =0$ 
for any $(i,k) \in \vG'_{\Fg}(\Bm)$.

\remark 
We should define the map $\Psi_t^\la$ for each  equivalence class $\CT_0(\la)[t]$ of $\CT_0(\la)$ 
since it may happen that 
$\Psi_t^\la (T) = \Psi_{t'}^\la(T')$ 
for different equivalence classes 
$\CT_0(\la)[t]$ and $\CT_0(\la)[t']$. 
For an example, 
put 
\begin{align*}
&T= 
\left( 
\,\,
\begin{array}{|c|c|c|}
\hline 
(1,1)&(1,1)&\mathbf{(1,2)}
\\
\hline
(2,1)&\mathbf{(1,3)}
\\
\cline{1-2}
\end{array}
\,,\,
\begin{array}{|c|c|c|}
\hline 
(1,2)&(2,2)&(1,3)
\\
\hline
(2,2)
\\
\cline{1-1}
\end{array}
\,,\,
\begin{array}{|c|}
\hline 
(1,3)
\\
\hline
(2,3)
\\
\cline{1-1}
\end{array}
\,\,
\right) 
\in \CT_0(\la)[t], 
\\[3mm]
&T'= 
\left( 
\,\,
\begin{array}{|c|c|c|}
\hline 
(1,1)&(1,1)&\mathbf{(1,3)}
\\
\hline
(2,1)&\mathbf{(1,2)}
\\
\cline{1-2}
\end{array}
\,,\,
\begin{array}{|c|c|c|}
\hline 
(1,2)&(2,2)&(1,3)
\\
\hline
(2,2)
\\
\cline{1-1}
\end{array}
\,,\,
\begin{array}{|c|}
\hline 
(1,3)
\\
\hline
(2,3)
\\
\cline{1-1}
\end{array}
\,\,
\right) 
\in \CT_0(\la)[t'].  
\end{align*}
Then 
we have 
\[\Psi_t^\la(T)= \Psi_{t'}^\la(T')= 
	\big( \bo{1} \otimes \bo{1} \otimes \bo{2} \big) 
	\boxtimes 
	\big( \bo{2} \otimes \bo{1} \otimes \bo{2} \otimes \bo{1} \big)
	\boxtimes 
	\big(\bo{1} \otimes \bo{2} \otimes \bo{1} \otimes \bo{1} \big)
.
\]

Now, we have the following theorem. 

\begin{thm}\
\label{Theorem b_la_mu}
\begin{enumerate}
\item 
For $\la, \mu \in \vL_{n,r}^+(\Bm)$, 
we have 
\[
\b_{\la\mu} = \sharp \big\{ T \in \CT_0(\la,\mu) \bigm| T \text{ : $U_q(\Fg) $-singular} \big\}.
\] 
\item 
$U_q(\Fg)$-crystal  structure  on $T_0(\la)$ 
is isomorphic to 
the $U_q(\Fg)$-crystal basis of  $W(\la)$ 
as crystals.  
\end{enumerate}
\end{thm}

\begin{proof}
We prove (\roi) by an induction for dominance order \lq\lq $\geq$" on $\vL_{n,r}^+(\Bm)$. 
First, 
we assume that 
$\CT_0(\la,\mu) \not= \emptyset$ 
and 
$\CT_0(\la,\nu) = \emptyset$ 
for any $\nu$ such that 
$\zeta (\nu) = \zeta (\mu)$ and $\nu > \mu$. 
For $(i,k) \in \vG'_{\Fg}(\Bm)$,  
one see easily that $\wt{e}_{(i,k)} \cdot T \in \CT_0(\la, \mu + \a_{(i,k)})$, 
$\zeta (\mu + \a_{(i,k)}) = \zeta (\mu)$ 
and 
$\mu + \a_{(i,k)}  > \mu$. 
Then,  
for any $T \in \CT_0(\la,\mu)$ and any $(i,k) \in \vG'_{\Fg}(\Bm)$, 
we have 
$\wt{e}_{(i,k)} \cdot T =0$ by the assumption. 
Thus, we have $\sharp \CT_0(\la,\mu) =  \sharp\{ T \in \CT_0(\la,\mu) \,|\, T \text{ : $U_q(\Fg) $-singular}\}$.  
Combining with  Lemma \ref{Lemma properties b_la_mu} (\roiv), 
we have 
$\b_{\la\mu} = \sharp\{ T \in \CT_0(\la,\mu) \,|\, T \text{ : $U_q(\Fg) $-singular}\}$. 

Next, as the assumption of the induction, 
we assume the claim of (\roi) 
for $\nu \in \vL_{n,r}^+(\Bm)$ 
such that $\zeta (\nu)= \zeta (\mu)$ and $\nu > \mu$.  
It is clear that 
\[ 
	\dim W(\la)_\mu
		= \sum_{\nu \in \vL_{n,r}^+(\Bm) \atop \zeta (\nu )= \zeta (\mu), \nu \geq \mu} 
		\b_{\la \nu} \cdot \dim \big( W(\nu^{(1)}) \boxtimes \cdots \boxtimes W(\nu^{(r)}) \big)_\mu.  
\]
Thus, we have 
\begin{align} 
\label{b_la_mu W(la) - some}
\b_{\la\mu} = \dim W(\la)_\mu 
	-  \sum_{\nu \in \vL_{n,r}^+(\Bm) \atop \zeta (\nu )= \zeta (\mu), \nu > \mu} 
		\b_{\la \nu} \cdot \dim \big( W(\nu^{(1)}) \boxtimes \cdots \boxtimes W(\nu^{(r)}) \big)_\mu. 
\end{align}
If 
$T \in \CT_0(\la,\mu)$ is not $U_q(\Fg)$-singular, 
there exists a sequence 
$(i_1,k_1), \cdots, (i_l, k_l) \in \vG'_{\Fg}(\Bm)$ 
such that 
$\wt{e}_{(i_1,k_1)} \cdots \wt{e}_{(i_l, k_l)} \cdot T \in \CT_0(\la, \mu + \a_{(i_1,k_1)} + \cdots + \a_{(i_l,k_l)} )$ 
is $U_q(\Fg)$-singular. 
Thus, by the assumption of the induction, we have 
\begin{align}
\label{some non singular}
	\sum_{\nu \in \vL_{n,r}^+(\Bm) \atop \zeta (\nu )= \zeta (\mu), \nu > \mu}
	& \b_{\la \nu} \cdot \dim \big( W(\nu^{(1)}) \boxtimes \cdots \boxtimes W(\nu^{(r)}) \big)_\mu 
\\ \notag
	&= 
	\sharp \big\{ T \in \CT_0(\la,\mu)  \bigm| \wt{e}_{(i,k)} \cdot T \not=0 \text{ for some } (i,k) \in \vG'_{\Fg} \big\}.
\end{align}
Since $\dim W(\la)_\mu = \CT_0(\la,\mu)$, 
\eqref{b_la_mu W(la) - some} and \eqref{some non singular} 
imply that 
$\b_{\la\mu} = \sharp \big\{ T \in \CT_0(\la,\mu) \bigm| T \text{ : $U_q(\Fg) $-singular} \big\}$. 

(\roii) follows from (\roi) and the definition of $\Psi_t^\la$. 
\end{proof}



\section{Some properties of the number  $\b_{\la\mu}$} 
In this section, 
we collect some properties of the number $\b_{\la\mu}$. 
For some extreme partitions, we have the following lemma. 

\begin{lem}\
\begin{enumerate} 
\item 
If $\la =((n), \emptyset, \cdots, \emptyset)$, 
\\
$\b_{\la\mu} = 
	\begin{cases} 
		1 & \text{if } \mu=((n_1), (n_2),  \cdots, (n_r)) \text{ for some } (n_1,\cdots, n_r) \in \ZZ_{\geq 0}^r
		\\
		0 &\text{otherwise} 
	\end{cases}
$. 

\item 
If $\la =((1^n), \emptyset, \cdots, \emptyset)$, 
\\
$\b_{\la\mu} = 
	\begin{cases} 
		1 & \text{if } \mu=((1^{n_1}), (1^{n_2}),  \cdots, (1^{n_r})) \text{ for some } (n_1,\cdots, n_r) \in \ZZ_{\geq 0}^r
		\\
		0 &\text{otherwise} 
	\end{cases}
$. 

\item 
If $\mu =(\emptyset, \cdots, \emptyset, (n))$, 
\\
$\b_{\la\mu} = 
	\begin{cases} 
		1 & \text{if } \la=((n_1), (n_2),  \cdots, (n_r)) \text{ for some } (n_1,\cdots, n_r) \in \ZZ_{\geq 0}^r
		\\
		0 &\text{otherwise} 
	\end{cases}
$. 

\item 
If $\mu =(\emptyset, \cdots, \emptyset, (1^n))$, 
\\
$\b_{\la\mu} = 
	\begin{cases} 
		1 & \text{if } \la=((1^{n_1}), (1^{n_2}),  \cdots, (1^{n_r})) \text{ for some } (n_1,\cdots, n_r) \in \ZZ_{\geq 0}^r
		\\
		0 &\text{otherwise} 
	\end{cases}
$. 
\end{enumerate}
\end{lem}

\begin{proof}
One can easily check them by using Lemma \ref{Lemma T singular} and Theorem \ref{Theorem b_la_mu}. 
\end{proof}

\para 
For  $r$-partitions 
$\la$ and $\mu$, 
we denote by 
$\la \supset \mu$ 
if  
$[\la] \supset [\mu]$. 
For $r$-partitions $\la$ and $\mu$ 
such that $\la \supset \mu$, 
we define the skew Young diagram by 
$\la/\mu = [\la] \setminus [\mu]$. 
One can naturally identify 
$\la/\mu$ with 
$(\la^{(1)}/\mu^{(1)},\cdots, \la^{(r)}/ \mu^{(r)})$,  
where 
$\la^{(k)}/ \mu^{(k)}$ ($1 \leq k \leq r$) 
is the usual skew Young diagram for $\la^{(k)} \supset \mu^{(k)}$. 
For a skew Young diagram $\la/\mu$, 
we  define a semi-standard tableau of shape $\la/\mu$ 
in a similar manner as in the case where the shape  is an $r$-partition. 
We denote by 
$\CT_0(\la/\mu,\nu)$ 
the set of semi-standard tableaux of shape $\la/\mu$ with weight $\nu$. 
Put 
$\CT_0(\la/\mu)=\bigcup_{\nu \in \vL_{n',r}(\Bm)} \CT_0(\la/\mu, \nu)$, 
where 
$n'=|\la/\mu|$. 
Then, 
we can describe the $U_q(\Fg)$-crystal structure on $\CT_0(\la/\mu)$ 
in a similar way as in the paragraphs \ref{def crystal T(la) start} - \ref{def crystal T(la) end}. 
Namely, 
we define the equivalence relation \lq\lq $\sim$" 
on $\CT_0(\la/\mu)$ 
in a similar way as in \ref{Definition equiv. on CT_0(la)}, 
and 
define the map 
$\Psi_t^{\la/\mu} : \CT_0(\la/\mu)[t] \ra \CB_{m_1}^{\otimes n_1} \boxtimes \cdots \boxtimes \CB_{m_k}^{\otimes n_r}$ 
for an each equivalence class $ \CT_0(\la/\mu)[t]$ of $\CT_0(\la/\mu)$ as in \ref{Definition reading}. 
Then we can show that $\Psi_t^{\la/\mu}$ is injective, 
and that 
$\Psi_t^{\la/\mu}(\CT_0(\la/\mu)[t]) \cup \{ 0 \}$ 
is stable under the Kashiwara operators $\wt{e}_{(i,k)}$, $\wt{f}_{(i,k)}$  for $(i,k) \in \vG'_{\Fg}(\Bm)$ 
(cf. Proposition \ref{Proposition reading}). 
Put 
\[ 
\CT_{sing}(\la/\mu,\nu) = \{T \in \CT_0(\la/\mu, \nu) \,|\, T \text{ : $U_q(\Fg)$-singular} \}. 
\]

From the tensor product rule for $U_q(\Fg)$-crystals,  
we have the following criterion on whether 
$T \in \CT_0(\la/\mu)$ 
is $U_q(\Fg)$-singular or not 
(note that $\CT_0(\la/\mu)=\CT_0(\la)$ if $\mu = \emptyset$). 

\begin{lem}
\label{Lemma T singular}
For $T \in \CT_0(\la/\mu)[t]$, 
let 
$\Psi_t^{\la/\mu}(T) =\big( \bo{i_1^{(1)}} \otimes \cdots \otimes \bo{i_{n_1}^{(1)}} \big)
	\boxtimes \cdots \boxtimes  \big( \bo{i_1^{(r)}} \otimes \cdots \otimes \bo{i_{n_r}^{(r)}} \big) $. 
Then,   
$T$ is $U_q(\Fg) $-singular 
if and only if 
the weight of $\big( \bo{i_1^{(k)}} \otimes \cdots \otimes \bo{i_{j}^{(k)}} \big) \in \CB_{m_k}^{\otimes j}$ 
is a partition (i.e. dominant integral weight of $\Fgl_{m_k}$) 
for any $1 \leq j \leq n_k$ and any  $1 \leq k \leq r$. 
\end{lem}

\begin{proof} 
It is clear that, 
for $T \in \CT_0(\la/\mu)[t]$ 
such that  
$\Psi_t^{\la/\mu}(T) =\big( \bo{i_1^{(1)}} \otimes \cdots \otimes \bo{i_{n_1}^{(1)}} \big)
	\boxtimes \cdots \boxtimes  \big( \bo{i_1^{(r)}} \otimes \cdots \otimes \bo{i_{n_r}^{(r)}} \big) $, 
$T$ 
 is $U_q(\Fg)$-singular 
if and only if 
$\big( \bo{i_1^{(k)}} \otimes \cdots \otimes \bo{i_{n_k}^{(k)}} \big) \in \CB_{m_k}^{\otimes n_k}$ 
is $U_q(\Fgl_{m_k})$-singular for any $k=1,\cdots,r$. 
Hence, 
the lemma follows from \cite[Lemma 6.1.1]{N} (see also \cite[Corollary 4.4.4]{HK}). 
\end{proof}

\para 
\label{Definition Bp}
Fix $\Bp=(r_1,\dots,r_g) \in \ZZ_{>0}^g$ such that $\sum_{k=1}^g r_k = r$. 
For $\la =(\la^{(1)},\cdots,\la^{(r)}) \in \vL_{n,r}^+(\Bm)$, 
put 
$\la^{[k]_{\Bp}} =(\la^{(p_k+1)},\cdots, \la^{(p_k+r_k)})$, 
where 
$p_k = \sum_{j=1}^{k-1} r_j$ with $p_1=0$.  
We define the map 
$\zeta^{\Bp} : \vL_{n,r}^+ (\Bm) \ra \ZZ_{\geq 0}^g$ 
by 
$\zeta^{\Bp}(\la) = (|\la^{[1]_{\Bp}}|, \cdots, |\la^{[g]_{\Bp}}|)$. 
Then, we have the following lemma.

\begin{lem}
For $\la,\mu \in \vL_{n,r}^+(\Bm)$ 
such that 
$\zeta^{\Bp}(\la) = \zeta^{\Bp}(\mu)$, 
we have 
\[
\b_{\la\mu} = \prod_{k=1}^g \b_{\la^{[k]_\Bp} \mu^{[k]_\Bp}}.
\]  
\end{lem}

\begin{proof}
It is enough to show the case where $\Bp=(r_1,r_2)$ 
since we can obtain the claim for general cases by the induction on $g$. 
If $\zeta^{\BP}(\la)= \zeta^{\Bp}(\mu)$ for $\Bp=(r_1,r_2)$,  
then 
we have the bijection 
\begin{align}
\label{bijection T_0 T_0p} 
\CT_0(\la,\mu) \ra \CT_0(\la^{[1]_\Bp}, \mu^{[1]_\Bp}) \times \CT_0(\la^{[2]_\Bp}, \mu^{[2]_\Bp}) 
\text{ such that } T \mapsto (T^{[1]_\Bp}, T^{[2]_\Bp}), 
\end{align} 
where 
$T^{[1]_\Bp} ((i,j,k)) = T((i,j,k))$ for $(i,j,k) \in [\la^{[1]_\Bp}]$, 
and 
$T^{[2]_\Bp}((i,j,k)) = (a, c-r_1)$ 
if $T((i,j,r_1+k))=(a,c)$ for $(i,j,k) \in [\la^{[2]_\Bp}]$. 
In this case, 
by the definition of $\Psi_t^\la$ and Lemma \ref{Lemma T singular}, 
it is clear that 
$T \in \CT_0(\la,\mu)$ is $U_q(\Fg)$-singular 
if and only if 
$T^{[1]_\Bp}$ (resp. $T^{[2]_\Bp}$) is $U_q(\Fg^{[1]})$-singular 
(resp. $U_q(\Fg^{[2]})$-singular), 
where 
$\Fg^{[1]}=\Fgl_{m_1} \oplus \cdots \oplus \Fgl_{m_{r_1}}$ 
(resp. $\Fg^{[2]}=\Fgl_{m_{r_1+1}} \oplus \cdots \oplus \Fgl_{m_{r}}$). 
Then, 
by Theorem \ref{Theorem b_la_mu} (\roi) together with \eqref{bijection T_0 T_0p}, 
we have 
$\b_{\la\mu} = \b_{\la^{[1]_\Bp} \mu^{[1]_\Bp}} \b_{\la^{[2]_\Bp} \mu^{[2]_\Bp}}$. 
\end{proof}

\para 
For $\la,\mu \in \vL_{n,r}^+(\Bm)$,   
we define the following set of sequences of $r$-partitions: 
\begin{align*}
	\Theta (\la,\mu) := 
		\Big\{ \la = & \la_{\lan r \ran}  \supset \la_{\lan r-1 \ran} \supset \cdots 
		\supset \la_{\lan 1 \ran} \supset \la_{\lan 0 \ran}=(\emptyset, \cdots, \emptyset) 
		\\& \Bigm| 
		(\la_{\lan k \ran})^{(k+1)} = \emptyset, \,\,  \,\, 
		|\la_{\lan k \ran} / \la_{\lan k-1 \ran}| = |\mu^{(k)}|  \text{ for }  k=1,\cdots,r 
	\Big\}. 
\end{align*}
It is clear that, 
for $\la_{\lan r \ran} \supset \cdots \supset \la_{\lan 0 \ran} \in \Theta (\la, \mu)$, 
we have that 
$\la_{\lan k \ran }=(\la_{\lan k \ran}^{(1)},\cdots, \la_{\lan k \ran}^{(k)}, \emptyset, \cdots, \emptyset)$, 
and that 
$|\la_{\lan k \ran}|=\sum_{j=1}^k |\mu^{(k)}|$. 
Then, 
we can rewrite Theorem  \ref{Theorem b_la_mu} (\roi) 
as the following corollary.

\begin{cor}
\label{Corollary b_la_mu}
For $\la,\mu \in \vL_{n,r}^+(\Bm)$, 
we have 
\begin{align}
\label{b_la_mu rewrite}
	\b_{\la\mu} 
	= 
	\sum_{\la_{\lan r \ran} \supset \cdots \supset \la_{\lan 0 \ran} \in \Theta(\la,\mu)} 
	\prod_{k=1}^r 
	\sharp \CT_{sing} \big( \la_{\lan k \ran} / \la_{\lan k-1 \ran} , (\emptyset, \cdots, \emptyset, \mu^{(k)}, \emptyset, \cdots, \emptyset) \big).
\end{align}
In particular, 
if $\la=(\emptyset, \cdots, \emptyset, \la^{(t)}, \emptyset, \cdots, \emptyset)$ for some $t$, 
then 
we have 
\begin{align}
\label{b_la_mu rewrite special}
\b_{\la \mu} = \sum_{\la_{\lan r \ran} \supset \cdots \supset \la_{\lan 0 \ran} \in \Theta(\la,\mu)} 
\prod_{k=1}^r \LR_{\la^{(t)}_{\lan k-1 \ran}, \mu^{(k)}}^{\la_{\lan k \ran}^{(t)}}, 
\end{align}
where 
$\LR_{\la^{(t)}_{\lan k-1 \ran}, \mu^{(k)}}^{\la_{\lan k \ran}^{(t)}}$ 
is the Littlewood-Richardson coefficient for $\la^{(t)}_{\lan k-1 \ran}$, $\mu^{(k)}$ and $\la_{\lan k \ran}^{(t)}$ 
with 
$\LR_{\emptyset, \emptyset}^{\emptyset} =1$.
\end{cor}

\begin{proof}
Note that we can identify the set $\Theta(\la,\mu)$ 
with the set of equivalence classes of $\CT_0(\la,\mu)$ with respect to the relation \lq\lq $\sim$" 
by 
corresponding 
$\la_{\lan r \ran } \supset \cdots \supset \la_{\lan 0 \ran} \in \Theta(\la,\mu)$ 
to the equivalence class of $\CT_0(\la,\mu)$ containing 
$T\in \CT_0(\la,\mu)$ 
such that 
$[\la_{\lan k \ran}] = \{(i,j,l) \in [\la] \,|\, T((i,j,l)) = (a,c) \text{ for some } 1 \leq a \leq m_c, 1 \leq c \leq k\}$ 
for any $k=1,\cdots, r$. 
Then 
Lemma \ref{Lemma T singular} and Theorem \ref{Theorem b_la_mu} (\roi) 
imply the equation \eqref{b_la_mu rewrite}. 

Assume that $\la=(\emptyset, \cdots, \emptyset, \la^{(t)}, \emptyset, \cdots, \emptyset)$ for some $t$. 
Then, 
for $\la_{\lan r \ran} \supset \cdots \supset \la_{\lan 0 \ran} \in \Theta(\la,\mu)$, 
we have 
\begin{align*}
\sharp \CT_{sing} \big( \la_{\lan k \ran} / \la_{\lan k-1 \ran} , (\emptyset, \cdots, \emptyset, \mu^{(k)}, \emptyset, \cdots, \emptyset) \big)
&= 
\sharp \CT_{sing} \big( \la_{\lan k \ran}^{(t)}/ \la_{\lan k-1 \ran}^{(t)} , \mu^{(k)} \big) 
\\
&= 
\LR^{\la_{\lan k \ran}^{(t)}}_{\la_{\lan k-1 \ran}^{(t)}, \mu^{(k)}}, 
\end{align*}
where 
the last equation follows from 
the original Littlewood-Richardson rule (\cite[Ch. I (9.2)]{Mac}).  
(Note that, 
for partitions $\la,\mu$  (not multi-partitions)  such that $\la \supset \mu$, 
the $U_q(\Fgl_m)$-crystal structure on $\CT_0(\la/\mu)$ 
does not depend on the choice of admissible reading 
(see \cite[Theorem 7.3.6]{HK}). 
Then a similar statement as in Lemma \ref {Lemma T singular} for $\CT_0(\la/\mu)$ 
under the Middle-Eastern reading 
coincides with the Littlewood-Richardson rule.)  
Then 
\eqref{b_la_mu rewrite} 
implies 
\eqref{b_la_mu rewrite special}.
\end{proof}

\remark 
\label{Remark LR rule}
In the case where 
$r=2$ and $\la=(\la^{(1)}, \emptyset)$, 
by \eqref{b_la_mu rewrite special}, 
we have 
\begin{align*}
\b_{\la\mu} 
&= \sum_{\la^{(1)}_{\lan 1 \ran}} 
	\LR^{\la^{(1)}}_{\la^{(1)}_{\lan 1 \ran} , \mu^{(1)}} \LR^{\la^{(1)}_{\lan 1 \ran}}_{\emptyset , \mu^{(2)}} 
\\
&= \LR^{\la^{(1)}}_{\mu^{(2)} , \mu^{(1)}}, 
\end{align*}
where the last equation follows from 
$\LR^{\la^{(1)}_{\lan 1 \ran}}_{\emptyset , \mu^{(2)}}  = \d_{\la^{(1)}_{\lan 1 \ran}, \mu^{(2)}}$. 
Thus, 
the Littlewood-Richardson coefficient $\LR^\la_{\mu,\nu}$ for partitions $\la,\mu,\nu$ 
is obtained as the number $\b_{(\la,\emptyset) (\mu,\nu)}$. 
Moreover, 
thanks to Lemma \ref{Lemma T singular} together with the reading $\Psi^{\la/\mu}_t$, 
we can regard \eqref{b_la_mu rewrite} as a generalization of the Littlewood-Richardson rule. 
We also remark the following classical fact: 
\begin{align} 
\label{Littlewood-Richardson in GL}
[\Res^{\mathbf{GL}_n}_{\mathbf{GL}_m \times \mathbf{GL}_{n-m}} V_\la : V_\mu \boxtimes V_\nu]_{\mathbf{GL}_m \times \mathbf{GL}_n} 
=\LR^\la_{\mu,\nu}, 
\end{align}
where 
$\mathbf{GL}_n$ (resp. $\mathbf{GL}_m$, $\mathbf{GL}_{n-m}$) 
is the general linear group of rank $n$ (resp. $m$, $n-m$), 
and 
$V_\la$ (resp. $V_\mu$, $V_\nu$) 
is the simple $\mathbf{GL}_n$-module (resp. simple $\mathbf{GL}_m$-module, simple $\mathbf{GL}_{n-m}$-module) 
corresponding to a partition $\la$ (resp. $\mu$, $\nu$). 
Comparing \eqref{def b_la_mu} with \eqref{Littlewood-Richardson in GL}, 
we may regard the number $\b_{\la\mu}$ 
as a generalization of Littlewood-Richardson coefficients.



\section{Characters of the Weyl modules and symmetric functions} 
\label{section character} 
\para 
For $\Bm=(m_1,\cdots, m_r) \in \ZZ_{>0}^r$, 
we denote by  
$\Xi_{\Bm} = \bigotimes_{k=1}^r \ZZ [ x_1^{(k)}, \cdots, x_{m_k}^{(k)}]^{\FS_{m_k}}$ 
the ring of symmetric polynomials 
(with respect to $\FS_{m_1}  \times \cdots \times \FS_{m_r}$) 
with variables $x_{i}^{(k)}$ ($1 \leq i \leq m_k$, $1 \leq k \leq r$). 
We denote by 
$x^{(k)}=(x_1^{(k)}, x_2^{(k)}, \cdots, x_{m_k}^{(k)})$ 
the set of $m_k$ independent variables for $k=1,\cdots,r$,  
and 
denote by $\bx=(x^{(1)}, \cdots, x^{(r)})$ 
the whole variables. 
Let 
$\Xi_{\Bm}^n$ 
be the subset of 
$\Xi_{\Bm}$ 
which consists of homogeneous symmetric polynomials of degree $n$. 
We also consider the inverse limit 
$\dis \Xi^n = \lim_{\stackrel{\longleftarrow}{\Bm}} \Xi^n_{\Bm}$
with respect to $\Bm$. 
Put 
$\Xi = \bigoplus_{n \geq 0} \Xi^n$. 
Then 
$\Xi$ becomes the ring of symmetric functions 
$\Xi=\bigotimes_{k=1}^r \ZZ[ X^{(k)}]^{\FS(X^{(k)})}$, 
where 
$X^{(k)}=(X_1^{(k)}, X_2^{(k)}, \cdots)$ 
is the set of (infinite) variables. 
We denote by 
$\bX=(X^{(1)},\cdots, X^{(r)})$ 
the whole variables of $\Xi$.

For $\la=(\la^{(1)},\cdots,\la^{(r)}) \in \vL_{n,r}^+(\Bm)$, 
put 
$S_\la (\bx) = \prod_{k=1}^r S_{\la^{(k)}}(x^{(k)})$ 
(resp. $S_\la(\bX) = \prod_{k=1}^r S_{\la^{(k)}}(X^{(k)})$),  
where 
$S_{\la^{(k)}}(x^{(k)})$ (resp. $S_{\la^{(k)}}(X^{(k)})$) 
is the  Schur polynomial (resp. Schur function) associated to $\la^{(k)}$ ($1 \leq k \leq r$) 
with variables $x^{(k)}$ (resp. $X^{(k)}$). 
Then 
$\{S_\la(\bx) \,|\, \la \in \vL_{n,r}^+(\Bm)\}$ 
(resp. $\{S_\la(\bX) \,|\, \la \in \vL_{n,r}^+\}$) 
gives a $\ZZ$-basis of $\Xi_{\Bm}^n$ (resp. $\ZZ$-basis of $\Xi^n$).

\para 
For an $\Sc_{n,r}(\vL_{n,r}(\Bm))$-module $M$, 
we define the character of $M$ by 
\[ \ch M = \sum_{\mu \in \vL_{n,r}(\Bm)} \dim M_\mu \cdot x^\mu \, \in \ZZ[\bx], \]
where 
$x^\mu = \prod_{k=1}^r (x_1^{(k)})^{\mu_1^{(k)}} (x_2^{(k)})^{\mu_2^{(k)}} \cdots (x_{m_k}^{(k)})^{\mu_{m_k}^{(k)}} $. 
Then the character of the Weyl module $W(\la)$ for $\Sc_{n,r}(\vL_{n,r}(\Bm))$ 
has the following properties. 
 
\begin{thm}\
\label{Theorem character of Weyl}
\begin{enumerate}
\item 
For $\la \in \vL_{n,r}^+(\Bm)$, we have 
\[
\ch W(\la) = \sum_{\mu \in \vL_{n,r}(\Bm)} 
	\left( \sum_{\nu \in \vL_{n,r}^+(\Bm)} \b_{\la\nu}\prod_{k=1}^r  K_{\nu^{(k)} \mu^{(k)}} \right) \cdot x^\mu, 
\]
where $ K_{\nu^{(k)} \mu^{(k)}}$ is the Kostka number corresponding to partitions $\nu^{(k)}$ and $\mu^{(k)}$.  
\item 
Put $\wt{S}_\la (\bx) =\ch W(\la)$ for $\la \in \vL_{n,r}^+(\Bm)$. 
Then, we have 
\[ \wt{S}_\la (\bx) = \sum_{\mu \in \vL_{n,r}^+(\Bm)} \b_{\la\mu} S_\mu(\bx). \]
\item 
$\{\wt{S}_{\la} (\bx) \,|\, \la \in \vL_{n,r}^+(\Bm)\}$ 
gives a $\ZZ$-basis of  $\Xi_{\Bm}^n$. 
\end{enumerate}
\end{thm}
\begin{proof}
Note Remark \ref {Remark weight cut}, 
we may assume that $m_k \geq n$ for any $k=1,\cdots, r$ 
by restricting the weights  if necessary for a general case.

Since there exists a bijection between a basis of $W(\la)_\mu$ and $\CT_0(\la,\mu)$, 
(\roi)  follows from \eqref{number semi-standard}.

It is known that 
\begin{align}
\label{Schur function}
S_\la(\bx) 
=
\sum_{\mu \in \vL_{n,r}(\Bm)} 
\dim \left( W(\la^{(1)}) \boxtimes \cdots \boxtimes W(\la^{(r)}) \right)_\mu \cdot x^\mu. 
\end{align}
Note that 
the $\mu$-weight space of an $\Sc_{n,r}$-module 
coincides with 
the $\mu$-weight space as the $U_q(\Fg)$-module via the homomorphism 
$\Phi_\Fg : U_q(\Fg) \ra \Sc_{n,r}$. 
Thus, 
the decomposition \eqref{def b_la_mu} together with \eqref{Schur function} 
implies (\roii). 

(\roiii) 
follows from (\roii) 
since the number $\b_{\la\mu}$ ($\la,\mu \in \vL_{n,r}^+(\Bm)$) 
has the uni-triangular property by Lemma \ref{Lemma properties b_la_mu}. 
\end{proof}

\para 
For $\la \in \vL_{n,r}^+(\Bm)$, 
let 
$\wt{S}_\la(\bX) \in \Xi^n$ 
be the image of 
$\wt{S}_\la(\bx)$ 
in the inverse limit. 
We denote by  
$\vL_{\geq 0,r}^+ = \bigcup_{n \geq 0} \vL_{n,r}^+$ 
the set of $r$-partitions. 
Then, 
Theorem \ref{Theorem character of Weyl} (\roiii) 
implies that 
$\{\wt{S}_\la(\bX) \,|\, \la \in \vL_{\geq 0 ,r}^+\}$ 
gives a $\ZZ$-basis of $\Xi$. 
For a certain extreme $r$-partition $\la$,  
$\wt{S}_\la(\bX)$ 
coincides with the Schur function as follows.

\begin{prop}
\label{Proposition wt{S} Schur function}
For $\la =(\la^{(1)},\cdots, \la^{(r)}) \in \vL_{n , r}^+$, 
if  $\la^{(l)} = \emptyset$ unless $l=t$ for some $t$, 
we have 
\[ 
	\wt{S}_\la (\bX) = S_{\la^{(t)}}(X^{(t)} \cup X^{(t+1)} \cup \cdots \cup X^{(r)}), 
\]
where 
 $S_{\la^{(t)}}(X^{(t)} \cup \cdots \cup X^{(r)}) \in \ZZ[X^{(t)} \cup \cdots \cup X^{(r)}]^{\FS(X^{(t)}\cup \cdots \cup X^{(r)})}$ 
 is the Schur function corresponding to the partition $\la^{(t)}$. 
\end{prop}

\begin{proof}
Assume that  $\la^{(l)} = \emptyset$ unless $l=t$, 
then 
we see that 
the variable $X_i^{(l)}$ ($i \geq 1$, $1 \leq l \leq t-1$) 
does not appear in $\wt{S}_\la(\bX)$ 
since 
$\la \geq \mu$ if 
$\dim W (\la)_\mu \not=0$. 
Note that 
we can regard 
$\ZZ[X^{(t)} \cup \cdots \cup X^{(r)}]^{\FS(X^{(t)}\cup \cdots \cup X^{(r)})}$ 
as 
a subring of 
$\Xi = \bigotimes_{k=1}^r \ZZ[X^{(k)}]^{\FS(X^{(k)})}$ in the natural way.   
By Theorem \ref{Theorem character of Weyl} (\roii) with 
\eqref{b_la_mu rewrite special}, 
we have 
\begin{align*}
\wt{S}_\la(\bX) 
&= 
	\sum_{\mu \in \vL_{n,r}^+} \left( \sum_{\la_{\lan r \ran } \supset \cdots \supset \la_{\lan 0  \ran} \in \Theta (\la,\mu)} 
		\prod_{k=1}^r \LR_{\la_{\lan k-1 \ran}^{(t)}, \mu^{(k)}}^{\la_{\lan k \ran}^{(t)}} \right) 
	S_\mu(\bX)
\\
&= 
	\sum_{\mu \in \vL_{n,r}^+}  
	\sum_{(\ast 1)} 
		\left( \prod_{k=t}^r \LR_{\la_{\lan k-1 \ran}^{(t)}, \mu^{(k)}}^{\la_{\lan k \ran}^{(t)}}  S_{\mu^{(k)}}(X^{(k)}) \right) 
\\
&= 
	\sum_{(\ast 2)}  
	\sum_{\mu \in \vL_{n,r}^+} 
		\left( \prod_{k=t}^r \LR_{\la_{\lan k-1 \ran}^{(t)}, \mu^{(k)}}^{\la_{\lan k \ran}^{(t)}}  S_{\mu^{(k)}}(X^{(k)}) \right) 
\\
&= 
	\sum_{(\ast 2)}  
		\prod_{k=t}^r 
			\left( \sum_{(\ast 3)} \LR_{\la_{\lan k-1 \ran}^{(t)}, \mu^{(k)}}^{\la_{\lan k \ran}^{(t)}}  S_{\mu^{(k)}}(X^{(k)}) \right) 
\\
&= 
	\sum_{(\ast 2)}  
		\prod_{k=t}^r 
			S_{\la^{(t)}_{\lan k \ran} / \la^{(t)}_{\lan k-1 \ran}}(X^{(k)}) 
	\qquad \big( \text{because of \cite[Ch. 1. (5.3)]{Mac}}\big)
\\
&= 
	S_{\la^{(t)}}(X^{(t)} \cup X^{(t+1)} \cup \cdots \cup X^{(r)}) 
	\qquad \big( \text{because of \cite[Ch. 1. (5.11)]{Mac}}\big), 
\end{align*}
where 
the summations 
$(\ast 1)$-$(\ast 3)$ 
run the following sets respectively: 
\begin{align*}
&(\ast 1) :  \big\{ \la^{(t)} = \la^{(t)}_{\lan r \ran} \supset \cdots \supset \la^{(t)}_{\lan t \ran} \supset \la^{(t)}_{\lan t-1 \ran} =\emptyset 
				\bigm|  |\la^{(t)}_{\lan k \ran}/ \la^{(t)}_{\lan k-1 \ran}|=|\mu^{(k)}| \text{ for } k=t, \cdots,r \big\},
\\
&(\ast 2) : \big\{ \la^{(t)} = \la^{(t)}_{\lan r \ran} \supset \cdots \supset \la^{(t)}_{\lan t \ran} \supset \la^{(t)}_{\lan t-1 \ran} =\emptyset \big\}, 
\\
&(\ast 3) : \big\{ \mu^{(k)} \text{ : partition}\big\}. 
\end{align*}
(In the above equations, note that $ \LR_{\la_{\lan k-1 \ran}^{(t)}, \mu^{(k)}}^{\la_{\lan k \ran}^{(t)}} =0$ unless 
$|\la_{\lan k \ran}^{(t)}|=|\la_{\lan k-1 \ran}^{(t)}|+| \mu^{(k)}|$.) 
\end{proof}
\para 
Thanks to the above lemma,  
the symmetric function 
$\wt{S}_\la(\bX)$ seems a generalization of  the Schur function. 

For 
$\la,\mu,\nu \in \vL_{\geq 0, r}^+$, 
we define the integer $c_{\la\mu}^{\nu} \in \ZZ$ by 
\[ \wt{S}_\la(\bX) \wt{S}_\mu(\bX) = \sum_{\nu \in \vL_{\geq 0, r}^+} c_{\la\mu}^\nu \wt{S}_\nu (\bX).\] 
Then we can compute the number $c_{\la\mu}^\nu$ as follows.

\begin{prop}
\label{Proposition properties c_la_mu^nu}
For $\la,\mu,\nu \in \vL_{\geq 0, r}^+$, 
we have the following. 
\begin{enumerate}
\item 
$c_{\la\mu}^\nu =0$ unless $|\nu|=|\la|+|\mu|$. 

\item 
Put  
$\big( \b'_{\t\nu} \big)_{\t, \nu \in \vL_{n,r}^+} = \big( \b_{\t \nu}\big)^{-1}_{\t, \nu \in \vL_{n,r}^+} $ $(n=|\nu|)$. 
Then we have 
\[ 
	c_{\la \mu}^\nu = \sum_{\xi, \eta,  \t \in \vL_{\geq 0,r}} \b_{\la\xi} \b_{\mu \eta} \b'_{\t\nu} \prod_{k=1}^r \LR_{\xi^{(k)} \eta^{(k)}}^{\t^{(k)}}.
\]

\item 
If 
$\zeta (\nu) = \zeta (\la+\mu)$, 
we have 
\[ 
	c_{\la\mu}^\nu = \prod_{k=1}^r \LR_{\la^{(k)} \mu^{(k)} }^{\nu^{(k)}}. 
\]

\item 
If $\la^{(l)}=\emptyset$ and $\mu^{(l)}=\emptyset $ unless $l=t$ for some $t$, 
we have 
\[ 
	c_{\la\mu}^\nu = 
		\begin{cases} 
			\LR_{\la^{(t)} \mu^{(t)} }^{\nu^{(t)}} & \text{ if $\nu^{(l)}=\emptyset $ unless $l=t$}, 
		\\
		0 & \text{otherwise}. 
		\end{cases}
\]
\end{enumerate}
\end{prop}
\begin{proof}
(\roi) is clear from the definitions. 
We prove (\roii). 
By Theorem \ref{Theorem character of Weyl} (\roii),  
we have 
\begin{align}
\label{c_la_mu^nu expand}
\wt{S}_\la(\bX) \wt{S}_\mu(\bX) 
&= 
	\Big( \sum_{\xi} \b_{\la \xi} S_{\xi}(\bX) \Big) \Big( \sum_{\eta} \b_{\mu \eta} S_{\eta}(\bX) \Big) 
\\
\notag
&= 
	\sum_{\xi, \eta} \b_{\la \xi} \b_{\mu \eta} S_\xi(\bX) S_{\eta}(\bX)
\\
\notag
&=\sum_{\xi, \eta}  \b_{\la \xi} \b_{\mu \eta} \left( \sum_{\t} \Big( \prod_{k=1}^r \LR^{\t^{(k)}}_{\xi^{(k)} \eta^{(k)}} \Big) S_{\t}(\bX) \right)
\\
\notag
&=\sum_{\xi, \eta}  \b_{\la \xi} \b_{\mu \eta} \left( \sum_{\t} \Big( \prod_{k=1}^r \LR^{\t^{(k)}}_{\xi^{(k)} \eta^{(k)}} \Big) 
	\Big(\sum_{\nu} \b'_{\t \nu} \wt{S}_\nu(\bX) \Big) \right)
\\
\notag 
&= 
\sum_{\nu} \left( \sum_{\xi, \eta,  \t} \b_{\la\xi} \b_{\mu \eta} \b'_{\t\nu} \prod_{k=1}^r \LR_{\xi^{(k)} \eta^{(k)}}^{\t^{(k)}} \right) 
\wt{S}_\nu(\bX).
\end{align} 
This implies (\roii). 
By Lemma \ref{Lemma properties b_la_mu} 
and the fact that $\LR^{\nu^{(k)}}_{\xi^{(k)} \eta^{(k)}} =0$ unless $|\nu^{(k)}|=|\xi^{(k)}|+|\eta^{(k)}|$, 
the equations \eqref{c_la_mu^nu expand} imply that 
\begin{align*}
&
\wt{S}_\la(\bX) \wt{S}_\mu(\bX) 
\\
&= 
\sum_{\nu  \atop \zeta(\nu) = \zeta(\la+\mu)} 
	\Big( \prod_{k=1}^r \LR^{\nu^{(k)}}_{\la^{(k)} \mu^{(k)}} \Big) S_{\nu}(\bX) 
	+\sum_{\nu \atop \zeta(\nu) \prec \zeta(\la+\mu)} 
		\Big( \sum_{\xi,\eta} \b_{\la \xi} \b_{\mu \eta} \prod_{k=1}^r \LR^{\nu^{(k)}}_{\xi^{(k)} \eta^{(k)}} \Big) S_{\nu}(\bX)
\\
&= 
	\sum_{\nu  \atop \zeta(\nu) = \zeta(\la+\mu)} 
		\Big( \prod_{k=1}^r \LR^{\nu^{(k)}}_{\la^{(k)} \mu^{(k)}} \Big) \wt{S}_{\nu}(\bX) 
	+\sum_{\nu \atop \zeta(\nu) \prec \zeta(\la+\mu)} a_{\la\mu}^\nu \wt{S}_\nu(\bX) 
	\qquad (a_{\la\mu}^\nu \in \ZZ).
\end{align*}
This implies (\roiii).

Finally, we prove (\roiv). 
By Proposition \ref{Proposition wt{S} Schur function}, 
we have 
\begin{align*}
\wt{S}_\la (\bX) \wt{S}_\mu(\bX) 
&= 
	S_{\la^{(t)}}(X^{(t)} \cup \cdots \cup X^{(r)}) S_{\mu^{(t)}}(X^{(t)} \cup \cdots \cup X^{(r)}) 
\\
&= 
	\sum_{\nu^{(t)}} \LR_{\la^{(t)} \mu^{(t)}}^{\nu^{(t)}} \, S_{\nu^{(t)}}(X^{(t)} \cup \cdot \cup X^{(r)}) 
\\
&= \sum_{\nu^{(t)}} \LR_{\la^{(t)} \mu^{(t)}}^{\nu^{(t)}} \,\wt{S}_{(\emptyset, \cdots, \emptyset, \nu^{(t)}, \emptyset, \cdots, \emptyset)} (\bX).
\end{align*}
This implies (\roiv).
\end{proof}

\para 
We have some conjectures for the number $c^\nu_{\la\mu}$ as follows. 
\begin{description}
\item[Conjecture 1] 
For $\la,\mu, \nu \in \vL_{\geq 0}^+$, 
the number $c_{\la\mu}^\nu $ is a non-negative integer. 
\end{description}
More strongly, we conjecture the following. 
\begin{description} 
\item[Conjecture 2] 
$c_{\la\mu}^\nu = \prod_{k=1}^r \LR_{\la^{(k)} \mu^{(k)}}^{\nu^{(k)}}$. 
\end{description}
Note that 
$\LR_{\la^{(k)} \mu^{(k)}}^{\nu^{(k)}} =0$ 
if $|\nu^{(k)}| \not= |\la^{(k)}|+|\mu^{(k)}|$, 
then 
Conjecture 2 
is equivalent to 
$c_{\la\mu}^\nu =0$ unless 
$\zeta(\nu)=\zeta(\la+\mu)$
by Proposition \ref{Proposition properties c_la_mu^nu} (\roiii). 

We remark that 
Conjecture 2 is true 
for $\la,\mu \in \vL_{\geq 0,r}^+$ 
such that 
$\la^{(l)}=\emptyset$ and $\mu^{(l)}=\emptyset $ unless $l=t$ for some $t$ 
by Proposition \ref{Proposition properties c_la_mu^nu} (\roiv).



\section{Decomposition matrices of cyclotomic $q$-Schur algebras}   
In this section, 
we consider the specialized cyclotomic $q$-Schur algebra 
$_F \Sc_{n,r}$ over a field $F$ with parameters 
$q, Q_1, \cdots,   Q_r \in F$ such that $q\not=0$. 
Hence, we omit the subscript $F$ for the objects over $F$. 
We also denote by 
$U_q(\Fg) = F \otimes_{\CA} \, _\CA U_q(\Fg)$ simply.  
Through this section, 
we assume that $m_k \geq n$ for any $k=1,\cdots,r$.

\para 
For $\Sc_{n,r}$-module $M$, 
we regard $M$ as a $U_q(\Fg)$-module through the homomorphism $\Phi_{\Fg}$. 
Then, by Lemma \ref{Lemma image Phi_g} (\roii), 
we see that 
a simple $U_q(\Fg)$-module appearing in the composition series of $M$ 
is the form 
$L(\la^{(1)}) \boxtimes \cdots \boxtimes L(\la^{(r)})$ 
($\la \in \vL_{n,r}^+(\Bm)$), 
where 
$L(\la^{(k)})$ is the simple $U_q(\Fgl_{m_k})$-module with  highest weight $\la^{(k)}$.

For a simple $\Sc_{n,r}$-module 
$L(\la)$ ($\la \in \vL_{n,r}^+(\Bm)$), 
let 
\[ 
x_{\la\mu} = [L(\la) : L(\mu^{(1)}) \boxtimes \cdots \boxtimes L(\mu^{(r)})]_{U_q(\Fg)}
\] 
be the multiplicity of 
$ L(\mu^{(1)}) \boxtimes \cdots \boxtimes L(\mu^{(r)})$ 
($\mu \in \vL_{n,r}^+(\Bm)$) 
in the composition series of $L(\la)$ as $U_q(\Fg)$-modules 
through $\Phi_{\Fg}$. 
Then we have the following lemma.

\begin{lem}\
\label{Lemma properties x_la_mu}
\begin{enumerate}
\item 
For $\la \in \vL_{n,r}^+(\Bm)$, 
$x_{\la\la}=1$. 

\item 
For $\la,\mu \in \vL_{n,r}^+(\Bm)$, 
if $x_{\la\mu} \not=0$, we have 
$\la \geq \mu$. 

\item 
For $\la,\mu \in \vL_{n,r}^+(\Bm)$,  
if  $\la \not=\mu$ and $\zeta(\la) = \zeta (\mu)$, 
we have $x_{\la\mu}=0$. 
\end{enumerate}
\end{lem}

\begin{proof}
By the definition of Weyl modules (see \ref{Definition Weyl}), 
we have 
$W(\la)=\Sc_{n,r}^- \cdot v_\la$, 
and 
$L(\la)$ is the unique simple top $W(\la)/ \rad W(\la)$ of $W(\la)$.  
Thus, 
by investigating the weights in $L(\la)$, 
we have (\roi) and (\roii). 

We prove (\roiii). 
We denote by $\ol{v}_\la$ the image of $v_\la$ 
under the natural surjection $W(\la) \ra L(\la)$. 
Then, we have 
$L(\la) = \Sc_{n,r}^- \cdot \ol{v}_\la$. 
One sees that 
\[
	M(\la) = \bigoplus_{\mu \in \vL_{n,r}(\Bm) \atop \zeta(\la) \succneqq \zeta(\mu)} 
		L(\la)_\mu
\]
is a $U_q(\Fg)$-submodule of $L(\la)$ 
since 
$\zeta (\mu \pm \a_{(i,k)})=\zeta (\mu)$ 
for any 
$(i,k) \in \vG_{\Fg}'(\Bm)$. 
It is clear that 
$M(\la)$ is also an $\Sc_{n,r}^-$-submodule of $L(\la)$, 
and 
$L(\la)/M(\la) = \Sc_{n,r}^- \cdot (\ol{v}_\la + M(\la))$. 
For $F_{(i_1,k_1)} F_{(i_2,k_2)} \cdots F_{(i_l,k_l)} \in \Sc_{n,r}^-$,  
if $i_j=m_{k_j}$ for some $j$, 
one sees that 
$F_{(i_1,k_1)}  \cdots F_{(i_l,k_l)} \cdot \ol{v}_\la \in M(\la)$.  
This implies that 
$L(\la)/M(\la)$ 
is generated by $\ol{v}_\la + M(\la)$ 
as a $U_q(\Fg)$-module, 
namely we have 
$L(\la)/M(\la) = U_q(\Fg) \cdot (\ol{v}_\la + M(\la))$. 
Hence, 
we have the surjective homomorphism of $U_q(\Fg)$-modules 
$\psi : L(\la)/M(\la) \ra L(\la^{(1)}) \boxtimes \cdots \boxtimes L(\la^{(r)})$ 
such that 
$\ol{v}_\la + M(\la) \mapsto \ol{v}_{\la^{(1)}} \boxtimes \cdots \boxtimes \ol{v}_{\la^{(r)}}$, 
where 
$\ol{v}_{\la^{(k)}}$ is a highest weight vector of $L(\la^{(k)})$ with the highest weight $\la^{(k)}$.  
We claim that $\psi$ is an isomorphism. 
If $\psi$ is not an isomorphism, 
there exists an element 
$x \in L(\la)_\mu$ 
such that $\la \not= \mu \in \vL_{n,r}^+(\Bm)$, $\zeta(\mu)= \zeta (\la)$ 
and $e_{(i,k)} \cdot x \in M(\la)$ for any $(i,k) \in \vG'_{\Fg}(\Bm)$, 
namely 
$x + M(\la) \in L(\la)/ M(\la)$ 
is a highest weight vector of highest weight $\mu$ as a $U_q(\Fg)$-module.  
On the other hand, 
we have 
$E_{(m_k,k)} \cdot x =0$ for $k=1,\cdots, r-1$ 
since $\zeta ( \mu + \a_{(m_k,k)}) \succ  \zeta (\mu) = \zeta (\la)$. 
Thus, we have that 
$E_{(i,k)} \cdot x \in M(\la)$ for any $(i,k) \in \vG'(\Bm)$.  
This implies  that 
$\Sc_{n,r} \cdot x $ 
is a proper $\Sc_{n,r}$-submodule of $L(\la)$ 
which contradict to the irreducibility of $L(\la)$ as an $\Sc_{n,r}$-module. 
Hence, $\psi$ is an isomorphism. 
Then, 
the isomorphism $L(\la)/M(\la) \cong  L(\la^{(1)}) \boxtimes \cdots \boxtimes L(\la^{(r)})$ 
together with  the definition of $M(\la)$ implies (\roiii). 
\end{proof}

\para 
For an algebra $\CA$, 
let $\CA \cmod$ be the category of finitely generated $\CA$-modules, 
and 
$K_0(\CA \cmod)$ 
be the Grothendieck group of $\CA \cmod$. 
For $M \in \CA \cmod$, 
we denote by $[M]$ the image of $M$ in $K_0(\CA \cmod)$. 

\para 
For 
$\la,\mu \in \vL_{n,r}^+(\Bm)$, 
let 
$d_{\la\mu} =[W(\la) : L(\mu)]_{\Sc_{n,r}}$ 
be the multiplicity of $L(\mu)$ in the composition series of $W(\la)$ as $\Sc_{n,r}$-modules, 
and 
$\ol{d}_{\la\mu} = [W(\la^{(1)}) \boxtimes \cdots \boxtimes W(\la^{(r)}) : L(\mu^{(1)}) \boxtimes \cdots \boxtimes L(\mu^{(r)})]_{U_q(\Fg)}$ 
be the multiplicity of 
$L(\mu^{(1)}) \boxtimes \cdots \boxtimes L(\mu^{(r)})$ 
in the composition series of 
$W(\la^{(1)}) \boxtimes \cdots \boxtimes W(\la^{(r)})$ 
as $U_q(\Fg)$-modules. 
Put 
\begin{align*}
&
D=\big( d_{\la\mu} \big)_{\la,\mu \in \vL_{n,r}^+(\Bm)}, 
& 
\ol{D}= \big( \ol{d}_{\la\mu} \big)_{\la,\mu \in \vL_{n,r}^+(\Bm)},  
\\
&
X=\big( x_{\la\mu} \big)_{\la,\mu \in \vL_{n,r}^+(\Bm)}, 
&
B=\big( \b_{\la\mu} \big)_{\la,\mu \in \vL_{n,r}^+(\Bm)}.
\end{align*}
Then the decomposition matrix $D$ of $\Sc_{n,r}$ 
is factorized as follows.

\begin{thm}
\label{Theorem factorization decom matrix}
We have that 
$B \cdot  \ol{D} = D \cdot X$. 
\end{thm}
\begin{proof}
By the definitions, 
for $\la \in \vL_{n,r}^+(\Bm)$, 
we have 
\begin{align*}
[W(\la)] 
&=
	\sum_{\mu \in \vL_{n,r}^+(\Bm)} d_{\la \mu} [L(\mu)] 
\\
&= 
	\sum_{\mu \in \vL_{n,r}^+(\Bm)} d_{\la \mu} 
		\Big( \sum_{\nu \in \vL_{n,.r}^+(\Bm)} x_{\mu \nu} [L(\nu^{(1)}) \boxtimes \cdots \boxtimes L(\nu^{(r)})] \Big) 
\\
&= 
	\sum_{\nu \in \vL_{n,r}^+} \Big( \sum_{\mu \in \vL_{n,r}^+(\Bm)} d_{\la\mu} x_{\mu\nu} \Big)   
		[L(\nu^{(1)}) \boxtimes \cdots \boxtimes L(\nu^{(r)})] 
\end{align*}
in $\CK_0(U_q(\Fg)\cmod)$. 
On the other hand, 
by taking a suitable modular system for $\Sc_{n,r}$, 
we have 
\begin{align*}
[W(\la)] 
&= 
	\sum_{\mu \in \vL_{n,r}^+(\Bm)} \b_{\la\mu} [W(\mu^{(1)}) \boxtimes \cdots \boxtimes W(\mu^{(r)})] 
\\
&= 
	\sum_{\mu \in \vL_{n,r}^+(\Bm)} \b_{\la\mu} 
		\Big( \sum_{\nu \in \vL_{n,r}^+(\Bm)} \ol{d}_{\mu \nu} [ L(\nu^{(1)}) \boxtimes \cdots \boxtimes L(\nu^{(r)})] \Big) 
\\
&= 
	\sum_{\nu \in \vL_{n,r}^+(\Bm)} \Big( \sum_{\mu \in \vL_{n,r}^+(\Bm)} \b_{\la\mu} \ol{d}_{\mu \nu} \Big) 
		[ L(\nu^{(1)}) \boxtimes \cdots \boxtimes L(\nu^{(r)})]
\end{align*} 
in $K_0(U_q(\Fg) \cmod)$. 
By comparing the coefficients of  $[ L(\nu^{(1)}) \boxtimes \cdots \boxtimes L(\nu^{(r)})]$, 
we obtain the claim of the theorem.  
\end{proof}

As a corollary of Theorem \ref{Theorem factorization decom matrix}, 
we have the product formula for decomposition numbers of $\Sc_{n,r}$ 
which has already  
obtained by \cite{Saw} in another method.

\begin{cor}
\label{Corollary prod formula}
For $\la, \mu \in \vL_{n,r}^+(\Bm)$ such that $\zeta(\la) = \zeta (\mu)$, 
we have 
\[ 
	d_{\la\mu} = \ol{d}_{\la\mu} = \prod_{k=1}^r d_{\la^{(k)} \mu^{(k)}}, 
\]
where 
$d_{\la^{(k)} \mu^{(k)}} =[W(\la^{(k)}) : L(\mu^{(k)})]$ 
is the decomposition number of $U_q(\Fgl_{m_k})$. 
\end{cor}

\begin{proof}
By Lemma \ref{Lemma properties b_la_mu} (\roii),  
for $\la, \mu, \nu \in \vL_{n,r}^+(\Bm)$, 
if 
$\b_{\la \nu} \ol{d}_{\nu \mu} \not=0$, 
then we have 
$\la \geq \nu \geq \mu$. 
Thus, 
if $\zeta(\la)=\zeta(\mu)$, 
we have 
\begin{align*}
\sum_{\nu \in \vL_{n,r}^+(\Bm)} \b_{\la\nu} \ol{d}_{\nu\mu} 
= 
\sum_{\nu \in \vL_{n,r}^+(\Bm) \atop \zeta(\la) = \zeta (\nu) = \zeta (\mu)} \b_{\la\nu} \ol{d}_{\nu\mu} 
= 
\ol{d}_{\la\mu},
\end{align*}
where the last equation follows from Lemma \ref{Lemma properties b_la_mu} (\roi) and (\roiii). 
Similarly, by using Lemma \ref{Lemma properties x_la_mu}, 
we see that 
$\sum_{\nu \in \vL_{n,r}^+(\Bm)} d_{\la \nu} x_{\nu\mu} =d_{\la\mu}$. 
Hence, 
Theorem \ref{Theorem factorization decom matrix} 
implies the claim of the corollary.  
\end{proof}

\remark 
\label{remark product formula}
In \cite{SW}, 
we also obtained the product formulas for decomposition numbers of $\Sc_{n,r}$ 
which are natural generalization of one in \cite{Saw} as follows.
Take $\Bp=(r_1,\cdots, r_g) \in \ZZ_{>0}^g$ such that $r_1+\cdots + r_g=r$ 
as in \ref{Definition Bp}. 
Then, 
for $\la,\mu \in \vL_{n,r}^+(\Bm)$ such that $\zeta^{\Bp}(\la)=\zeta^{\Bp}(\mu)$, 
we have 
\begin{align}
\label{prod formula}
 d_{\la \mu} = \prod_{k=1}^g d_{\la^{[k]_{\Bp}} \mu^{[k]_{\Bp}}} 
 \end{align}
by \cite[Theorem 4.17]{SW}, 
where 
$d_{\la^{[k]_{\Bp}} \mu^{[k]_{\Bp}}}$ 
is the decomposition number of $\Sc_{n_k, r_k}$ ($n_k=|\la^{[k]}|$) 
with parameters $q, Q_{p_k+1},\cdots, Q_{p_k +r_k}$. 
However, 
the formula \eqref{prod formula} for general $\Bp$ ($\not= (1, \cdots,1)$) 
is not obtained in a similar way as in Corollary \ref{Corollary prod formula} 
since 
$\ol{\Sc}_{n,r}^\Bp$  
does not realize as a subalgebra of $\Sc_{n,r}$ 
in a similar  way as in Lemma \ref {Lemma image Phi_g}, 
where 
$\ol{\Sc}_{n,r}^\Bp$ 
is a subquotient algebra of $\Sc_{n,r}$ defined in \cite[2.12]{SW}. 
(Note that 
$\ol{\Sc}^{\Bp}_{n,r} \cong \bigoplus_{(n_1,\cdots, n_g) \atop n_1+\cdots +n_g=n}  \Sc_{n_1,r_1} \otimes \cdots \otimes \Sc_{n_g,r_g}$ 
by \cite[Theorem 4.15]{SW}. 
Thus, 
if $\Bp=(1,\cdots, 1)$, 
$\ol{\Sc}^\Bp_{n,r}$ coincides with the right-hand side of  the isomorphism in Lemma \ref {Lemma image Phi_g}.) 
Hence, 
in order to obtain the formula \eqref{prod formula} for general $\Bp$,  
it is essential to 
take the subquotient algebra $\ol{\Sc}^{\Bp}_{n,r}$  
as in \cite{SW}. 
\\

For  special parameters,   
we see that 
the matrix $X$ becomes the identity matrix as the following corollary. 

\begin{cor}\
\label{Corollary Q_1=Q_2= cdots =Q_r=0}
\begin{enumerate}
\item 
If $Q_1=Q_2= \cdots = Q_r=0$, 
the matrix $X$ is the identity matrix. 
In particular,  
we have 
$D= B \cdot \ol{D}$. 

\item 
If 
$q=1$, $Q_1= Q_2 = \cdots = Q_r$ (not necessary to be $0$), 
the matrix $X$ is the identity matrix. 
Moreover, we have 
$D=B$ 
if 
$\operatorname{char} F=0$. 
\end{enumerate}
\end{cor}

\begin{proof}
Assume that $Q_1=Q_2= \cdots = Q_r=0$. 
We denote by $E_{(i,k)}^{(c)} = 1 \otimes E_{(i,k)}^c/[c]!  $ 
	(resp. $F_{(i,k)}^{(c)} = 1 \otimes F_{(i,k)}^c/[c]!$) $\in F \otimes_{\CA} \, _\CA \Sc_{n,r} \cong \, _F \Sc_{n,r}$. 
By the triangular decomposition of $\Sc_{n,r}$, 
we have 
\[ 
	\s_{(i,k)}^\la = \sum r_{(i_1, k_1, c_1), \cdots, (i_l, k_l, c_l)}^{(i_1', k_1', c_1'), \cdots, (i_l', k_l', c_l')} 
	F_{(i_1',k_1')}^{(c_1')} \cdots F_{(i_l', k_l')}^{(c_l')} E_{(i_1, k_1)}^{(c_1)} \cdots E_{(i_l,k_l)}^{(c_l)} 1_\la, 
\] 
for some  $ r_{(i_1, k_1, c_1), \cdots, (i_l, k_l, c_l)}^{(i_1', k_1', c_1'), \cdots, (i_l', k_l', c_l')}  \in F$. 
First,  we prove the following claim. 
\begin{description} 
\item[Claim A] 
If 
$ r_{(i_1, k_1, c_1), \cdots, (i_l, k_l, c_l)}^{(i_1', k_1', c_1'), \cdots, (i_l', k_l', c_l')}  \not=0$ 
and 
$F_{(i_1',k_1')}^{(c_1')} \cdots F_{(i_l', k_l')}^{(c_l')} E_{(i_1, k_1)}^{(c_1)} \cdots E_{(i_l,k_l)}^{(c_l)} 1_\la \not=0$, 
then 
we have 
$\zeta (\la + c_1 \a_{(i_1,k_1)} + \cdots + c_l \a_{(i_l,k_l)}) \succ \zeta (\la)$. 
\end{description}
Note that $Q_1=Q_2= \cdots = Q_r=0$, 
we see easily that 
$\He_{n,r}$ is a $\ZZ/ r \ZZ$-graded algebra 
with $\deg (T_0)=\ol{1}$ and $\deg (T_i)= \ol{0}$, 
where we put $\ol{k} = k + r \ZZ \in \ZZ/ r \ZZ$ for $k \in \ZZ$. 
We can also check that 
$m_\la$ ($\la \in \vL_{n,r}(\Bm)$) is a homogeneous element of $\He_{n,r}$. 
Since 
$\s_{(i,k)}^\la (m_\la) = m_{\la}\cdot (L_{N+1}+ \cdots + L_{N+ \la_{i}^{(k)}})$ 
($N= \sum_{l=1}^{k-1} |\mu^{(l)}|+ \sum_{j=1}^{i-1}\mu_j^{(k)}$),  
we have 
$\deg (\s_{(i,k)}^\la (m_\la)) = \deg (m_\la) + \ol{1}$. 
On the other hand, 
by \cite[Lemma 6.10]{W}, 
we see that 
$F_{(i_1',k_1')}^{(c_1')} \cdots F_{(i_l', k_l')}^{(c_l')} E_{(i_1, k_1)}^{(c_1)} \cdots E_{(i_l,k_l)}^{(c_l)} 1_\la 
(m_\la)$ 
is a homogeneous element of $\He_{n,r}$ 
with degree $\deg(m_\la)$ 
if $i_j \not=m_{k_j}$ and $i'_j \not=m_{k'_j}$ for any $j=1,\cdots, l$. 
Thus, 
if  
$ r_{(i_1, k_1, c_1), \cdots, (i_l, k_l, c_l)}^{(i_1', k_1', c_1'), \cdots, (i_l', k_l', c_l')}  \not=0$ 
and 
$F_{(i_1',k_1')}^{(c_1')} \cdots F_{(i_l', k_l')}^{(c_l')} E_{(i_1, k_1)}^{(c_1)} \cdots E_{(i_l,k_l)}^{(c_l)} 1_\la \not=0$, 
then there exists $j$ such that $i_j = m_{k_j}$, 
and 
this implies that 
$\zeta (\la + c_1 \a_{(i_1,k_1)} + \cdots + c_l \a_{(i_l,k_l)}) \succ \zeta (\la)$. 
Now, we proved Claim A. 

We have already shown that 
$x_{\la\la}=1$, 
and 
$x_{\la\mu}=0$ for  
$\la \not=\mu$ such that $\zeta(\la) = \zeta (\mu)$ 
in Lemma \ref{Lemma properties x_la_mu}. 
Thus, it is enough to show that 
$x_{\la\mu}=0$ for $\la,\mu \in \vL_{n,r}^+(\Bm)$ such that $\zeta (\la) \not= \zeta (\mu)$.

Suppose that 
$x_{\la\mu}\not=0$ for some  $\la,\mu \in \vL_{n,r}^+(\Bm)$ such that $\zeta (\la) \not= \zeta (\mu)$. 
We recall that 
$L(\la) = \Sc_{n,r}^- \cdot \ol{v}_\la$, 
where $\ol{v}_\la = v_\la + \rad W(\la) \in W(\la) / \rad W(\la) \cong L(\la)$.  
Then, 
it is clear that 
$L(\la)_\mu \not=0$. 
This implies the existence of a non-zero element 
\[ 
v'=\sum r_{(i_1,k_1), \cdots, (i_c,k_c)} F_{(m_{k'},k')} F_{(i_1,k_1)} \cdots F_{(i_c,k_c)} \cdot \ol{v}_\la \in L(\la) 
\quad 
( r_{(i_1,k_1), \cdots, (i_c,k_c)} \in F) 
\]
such that 
$E_{(i,k)} \cdot v'=0$ for any $(i,k) \in \vG'_{\Fg}(\Bm)$, 
where 
 the summation runs 
\[
\{((i_1,k_1), \cdots, (i_c,k_c)) \in (\vG'_{\Fg}(\Bm))^c \,|\, \a_{(i_1, k_1)} + \cdots + \a_{(i_c, k_c)} =\a\}
\]
for some $\a \in \bigoplus_{(i,k) \in \vG'_{\Fg}(\Bm)} \ZZ \a_{(i,k)}$. 
Namely 
$v'$ is a $U_q(\Fg)$-highest weight vector of highest weight $ \mu'= \la - \a - \a_{(m_{k'},k')}$. 
It is clear that 
$\zeta(\la) = \zeta (\la-\a)$. 
Since 
$E_{(m_{k},k)}$ 
($k \not= k'$) 
commute with 
$F_{(m_{k'},k')}$ and $F_{(j,l)}$ ($(j,l) \in \vG'_{\Fg}(\Bm)$), 
we have that 
$E_{(m_{k},k)} \cdot v'=0$ 
for any $k \in \{1,\cdots, r-1\} \setminus \{k' \}$. 
On the other hand,  
for $((i_1,k_1), \cdots, (i_c,k_c)) \in (\vG'_{\Fg}(\Bm))^c$ such that $\a_{(i_1, k_1)} + \cdots + \a_{(i_c, k_c)} =\a$, 
we have 
\begin{align}
&
\label{E_{m_{k'} k'} vanish}
E_{(m_{k'},k')} F_{(m_{k'},k')} F_{(i_1,k_1)} \cdots F_{(i_c,k_c)} \cdot \ol{v}_\la 
\\ \notag
&=\Big\{
	F_{(m_{k'},k')} E_{(m_{k'},k')}  
	+
	\Big( q^{(\la-\a)_{m_{k'}}^{(k')} - (\la-\a)_1^{(k'+1)}} 
			(q^{-1} \s_{(m_{k'},k')}^{\la -\a}  - q \s_{(1,k'+1)}^{\la -\a} ) 1_{\la -\a}  \Big)
	\Big\} \cdot 
	\\ \notag 
	&\qquad 
	1_{\la - \a}  
	F_{(i_1,k_1)} \cdots F_{(i_c,k_c)} \cdot \ol{v}_\la.
\end{align}
Note that $\zeta (\la- \a)= \zeta (\la)$, 
\eqref{E_{m_{k'} k'} vanish} together with Claim A  
implies that 
\begin{align*}
E_{(m_{k'},k')} F_{(m_{k'},k')} F_{(i_1,k_1)} \cdots F_{(i_c,k_c)} \cdot \ol{v}_\la 
=0.  
\end{align*}
Thus, we have 
\begin{align*}
E_{(m_{k'},k')} \cdot v' 
&= 
\sum r_{(i_1,k_1), \cdots, (i_c,k_c)}  E_{(m_{k'},k')} F_{(m_{k'},k')} F_{(i_1,k_1)} \cdots F_{(i_c,k_c)} \cdot \ol{v}_\la 
=0.
\end{align*} 
As a consequence, 
we have that 
$E_{(i,k)} \cdot v' = 0$ for any $(i,k) \in \vG'(\Bm)$, 
and this implies that 
$\Sc_{n,r} \cdot v'$ is a proper $\Sc_{n,r}$-submodule of $L(\la)$. 
However, 
this contradicts to the irreducibility  of  $L(\la)$ as $\Sc_{n,r}$-module. 
Thus, we have that 
$x_{\la\mu}=0$ for $\la,\mu \in \vL_{n,r}^+(\Bm)$ such that $\zeta (\la) \not= \zeta (\mu)$. 
Now we proved (\roi). 

Next we prove (\roii). 
Let $\He_{n,r}$ (resp. $\He'_{n,r}$) 
be the Ariki-Koike algebra over $F$ with parameters 
$q=1, Q_1= \cdots = Q_r=0$ 
(resp. $q=1, Q'_1= \cdots = Q'_r = Q' \not=0$), 
and 
$\Sc_{n,r}$ (resp. $\Sc_{n,r}'$) 
be the cyclotomic $q$-Schur algebra 
associated to 
$\He_{n,r}$ (resp. $\He'_{n,r}$). 
We denote by $T_0, T_1, \cdots, T_{n-1}$ 
(resp. $T'_0, T'_1, \cdots, T'_{n-1}$) 
the generators of $\He_{n,r}$ (resp. $\He'_{n,r}$) 
as in \ref{Definition Ariki-Koike}. 
Then 
we can check that 
there exists an isomorphism 
$\phi: \He_{n,r} \ra \He'_{n,r}$ 
such that 
$\phi(T_0) = T_0' -  Q'$ 
and 
$\phi(T_i)=T'_i$ ($1 \leq i \leq n-1$). 
We can also check that 
$M^\mu \cong M'^\mu$ for $\mu \in \vL_{n,r}(\Bm)$ under the isomorphism $\phi$,  
where 
$M^\mu$ (resp. $M'^\mu$) 
is the right $\He_{n,r}$-module  (resp. $\He'_{n,r}$-module) 
defined in \ref{Definition M^mu}. 
Thus, 
we have 
$\Sc_{n,r} \cong \Sc'_{n,r}$ 
as algebras. 
Then  
(\roi) implies (\roii) 
since 
$\ol{D}$ is the identity matrix when $q=1$ 
if $\operatorname{char} F=0$. 
\end{proof}

\remarks 

(\roi) 
In Theorem \ref{Theorem factorization decom matrix}, 
the matrix $B \cdot \ol{D}$ 
does not depend on the choice of parameters $Q_1,\cdots,Q_r$. 

(\roii) 
If $\Sc_{n,r}$ is semi-simple, 
both of $D$ and $\ol{D}$ are identity matrices.  
Thus, we have $B=X$. 

(\roiii) 
By  Theorem \ref{Theorem factorization decom matrix},  
for $\la,\mu \in \vL_{n,r}^+$, 
we have 
\begin{align*}
&
d_{\la\mu} + x_{\la \mu} 
= \sum_{\nu \in \vL_{n,r}^+} \b_{\la\nu} \ol{d}_{\nu\mu} - \sum_{\nu \in \vL_{n,r}^+ \atop \la > \nu > \mu } d_{\la\nu} x_{\nu\mu}. 
\end{align*}
Thus, 
we see that 
the matrix $B \cdot \ol{D}$ 
gives an upper bound of 
both $d_{\la\mu}$ and $x_{\la\mu}$.



\section{The Ariki-Koike algebra as a subalgebra of  $\Sc_{n,r}$} 
In this section, 
we consider the algebras 
over an commutative ring $R$ 
with parameters $q, Q_1, \cdots, Q_r \in R$ such that $q$ is invertible in $R$.   
Hence, we omit the subscript $R$ for the objects over $R$. 
\para 
For 
$\mu \in \vL_{n,r}(\Bm)$, 
put 
\begin{align*}
&
X^\mu_{\mu+\a_{(i,k)}} =\big\{ 1, s_{N+1}, (s_{N+1}, s_{N+2}), \cdots, (s_{N+1} s_{N+2} \cdots s_{N+\mu_{i+1}^{(k)}-1}) \big\}, 
\\
&
X^\mu_{\mu -\a_{(i,k)}} =\big\{ 1, s_{N-1}, (s_{N-1}, s_{N-2}), \cdots, (s_{N-1} s_{N-2} \cdots s_{N-\mu_{i}^{(k)} +1}) \big\}, 
\end{align*}
where 
$s_j=(j,j+1) \in \FS_n$ is the adjacent transposition, 
and 
$N= \sum_{l=1}^{k-1} |\mu^{(l)}| + \sum_{j=1}^i \mu_j^{(k)}$. 
Then, 
by \cite[Lemma 6.10, Proposition 7.7 and Theorem 7.16 (\roi)]{W},  
we have 
\begin{align}
&\label{generator in Sc 1}
1_\nu (m_\mu) = \d_{\mu, \nu} m_\mu, 
\\
& \label{generator in Sc 2}
e_{(i,k)} (m_\mu) = q^{- \mu_{i+1}^{(k)} +1 } m_{\mu + \a_{(i,k)}} \Big( \sum_{y \in X_{\mu+ \a_{(i,k)}}^{\mu}} q^{\ell(y)} T_y \Big), 
\\
& \label{generator in Sc 3}
f_{(i,k)} (m_\mu) = q^{- \mu_{i}^{(k)} +1 } m_{\mu - \a_{(i,k)}} h_{-(i,k)}^\mu \Big( \sum_{x \in X_{\mu - \a_{(i,k)}}^{\mu}} q^{\ell(x)} T_x \Big), 
\end{align}
where 
$h_{-(i,k)}^\mu = \begin{cases} 1 & (i \not= m_k), \\ L_N - Q_{k+1} & (i=m_k) \,\, (N:=|\mu^{(1)}| + \cdots + |\mu^{(k)}|). \end{cases}$

\para 
Put  
$\w =(\emptyset, \cdots, \emptyset, (1^n)) \in \vL_{n,r}^+(\Bm)$. 
Then, 
it is well known that 
$M^\w \cong \He_{n,r}$ as right $\He_{n,r}$-modules, 
and that 
$1_\w \Sc_{n,r} 1_\w = \End_{\He_{n,r}} (M^\w, M^\w) \cong \He_{n,r}$ 
as $R$-algebras. 
Put 
$C_0 = 1_{\w} f_{(m_{r-1}, r-1)} e_{(m_{r-1},r-1)} 1_\w$, 
$C_i = 1_{\w} f_{(i,r)} e_{(i,r)} 1_\w \in \Sc_{n,r}$ for $i=1,\cdots,n-1$.
Then, we can realize $\He_{n,r}$ as a subalgebra of $\Sc_{n,r}$ as the following proposition.

\begin{prop}\
\label{Proposition AK in Sc}
\begin{enumerate}
\item 
The subalgebra of $\Sc_{n,r}$ 
generated by 
$C_{0}, C_1, \cdots, C_{n-1}$ 
is isomorphic to the Ariki-Koike algebra $\He_{n,r}$. 
Moreover, 
the subalgebra of $\Sc_{n,r}$ 
generated by 
$ C_1, \cdots, C_{n-1}$ 
is isomorphic to the Iwahori-Hecke algebra $\He_{n}$ of symmetric group $\FS_n$. 

\item 
Under the isomorphism 
$1_\w \Sc_{n,r} 1_\w \cong \He_{n,r}$, 
we have 
$T_0 =C_0 + Q_r 1_\w$, 
$T_i = C_i - q^{-1}1_\w$. 

\end{enumerate}
\end{prop}

\begin{proof}
It is clear that 
$C_{0}, C_1, \cdots, C_{n-1}$ 
are elements of $1_\w \Sc_{n,r} 1_\w$. 
We remark that 
the isomorphism 
$\End_{\He_{n,r}} (M^\w, M^\w) \cong \He_{n,r}$ 
is given by 
$\vf \mapsto \vf(m_\w)$ (note that $m_\w =1$). 
Moreover, 
by \eqref{generator in Sc 1} - \eqref{generator in Sc 3}, 
we have 
\begin{align*}
C_0(m_\w) 
&= 
1_{\w} f_{(m_{r-1}, r-1)} e_{(m_{r-1},r-1)} 1_\w (m_\w) 
\\
&= m_{\w} (L_1 - Q_r). 
\end{align*}
Since $m_\w =1$ and $L_1=T_0$, 
we have $C_0 (m_\w) = T_0 -Q_r$. 
Similarly, 
we have 
$C_i(m_\w) = T_i + q^{-1}$ 
for $i=1,\cdots,n-1$.  
Thus, 
$\He_{n,r}$ 
is generated by 
$C_0, C_1,\cdots,C_{n-1}$ 
under the isomorphism $1_\w \Sc_{n,r} 1_\w \cong \He_{n,r}$, 
and 
$\He_n$ 
is generated by 
$C_1,\cdots,C_{n-1}$. 
Now, (\roii) is clear. 
\end{proof}

\para 
Let 
$\CF=\Hom_{\Sc_{n,r}} (\Sc_{n,r} 1_\w, -) : \Sc_{n,r} \cmod \ra \He_{n,r} \cmod$ 
be the Schur functor. 
It is well known that, for $M \in \Sc_{n,r}\cmod$, 
$\CF(M) = 1_\w M$  under the isomorphism $1_\w \Sc_{n,r} 1_\w  \cong \He_{n,r}$. 
It is also known that 
$\{1_\w L(\la) \not=0 \,|\, \la \in \vL_{n,r}^+(\Bm)\}$ 
gives a complete set of non-isomorphic simple $\He_{n,r}$-modules when $R$ is a field. 

Let $e$ be the smallest positive integer such that  
$1+ (q^2) + (q^2)^2 + \cdots (q^2)^{e-1}=0$. 
We say that 
a partition (not multi-partition) $\la=(\la_1, \la_2,\cdots)$ 
is $e$-restricted if $\la_i-\la_{i+1} < e$ 
for any $i \geq 1$.

As a corollary of  Corollary \ref{Corollary Q_1=Q_2= cdots =Q_r=0}, 
we have the following classification of simple $\He_{n,r}$-modules for special parameters.  
We remark that this classification has already proved by \cite[Theorem 1.6]{AM} and \cite[Theorem 3.7]{M} in other methods. 

\begin{cor}
\label{Corollary simple AK}
Assume that $R$ is a field, 
and that 
either 
$Q_1=Q_2=\cdots =Q_r=0$ 
or 
$q=1$, $Q_1=Q_2=\cdots=Q_r$.
Then 
$1_\w L(\la) \not=0$ 
if and only if 
$\la^{(k)}=\emptyset$ for $k<r$ and $\la^{(r)}$ is an $e$-restricted partition. 

\end{cor}

\begin{proof}
By Corollary \ref{Corollary Q_1=Q_2= cdots =Q_r=0}, 
we have that 
$1_\mu L(\la) \not=0$ only if 
$\zeta(\mu) = \zeta (\la)$. 
In particular, 
we have that 
$\la^{(k)}=0$ for any $k <r$ 
if 
$1_\w L(\la) \not=0$. 
On the other hand, 
$L(\la) \cong L(\la^{(1)}) \boxtimes \cdots \boxtimes L(\la^{(r)})$ 
as $U_q(\Fg)$-modules by Corollary \ref{Corollary Q_1=Q_2= cdots =Q_r=0}. 
In particular, 
when $\la^{(k)} = \emptyset$ for any $k <r$, 
we have that 
$L(\la) \cong L(\la^{(r)})$ as $U_q(\Fgl_{m_r})$-modules. 
Moreover,  it is well known that 
$1_\w L(\la^{(r)}) \not=0$ if and only if 
$\la^{(r)}$ is an $e$-restricted partition 
(\cite[Theorem 6.3, 6.8]{DJ}). 
These results imply the corollary. 
\end{proof}






\begin{thebibliography}{DJM10}
\bibitem[AM]{AM} 
S.~Ariki and A.~Mathas, 
\newblock The number of simple modules of the Hecke algebras of type $G(r,1,n)$, 
\newblock {\em Math. Z.} {\bf 233} (2000), 601-623. 

\bibitem[Du]{Du}
J.~Du, 
\newblock A note on quantized Weyl reciprocity at root of unity, 
\newblock{\em Algebra Colloq. } {\bf 2} (1995), 363-372. 

\bibitem[DJ]{DJ}
R.~Dipper and G.~James, 
\newblock Representations of Hecke algebras of general linear groups, 
\newblock{\em Proc. London Math. Soc. (3)} {\bf 52} (1998), 20-56.


\bibitem[DJM]{DJM98}
R.~Dipper, G.~James, and A.~Mathas, 
\newblock Cyclotomic {$q$}-{S}chur algebras, 
\newblock {\em Math. Z.} {\bf 229} (1998), 385-416.


\bibitem[HK]{HK} 
J.~Hong and S.-J.~Kang,  hIntroduction to Quantum Groups and Crystal Basesh, Grad. Stud. in Math. {\bf Vol. 42}, A.M.S. (2002). 


\bibitem[J]{J} 
M.~Jimbo, 
\newblock A q-analogue of $U(\mathfrak{gl}(N+1))$, Hecke algebra and the Yang-Baxter equation, 
\newblock {\em Lett. Math. Phys.}  {\bf 11} (1986), 247-252.  


\bibitem[KN]{KN} 
M.~Kashiwara and T.~Nakashima, 
\newblock Crystal Graphs for Representations of the $q$-Analogue of Classical Lie Algebras, 
\newblock{\em J. Algebra} {\bf 165} (1994), 295-345.  


\bibitem[Mac]{Mac}
I.G.~Macdonald, 
\newblock{Symmetric Functions and Hall Polynomials}, 
2nd edition, Clarendon Press, Oxford, 1995. 


\bibitem[M1]{M} 
A.~Mathas, 
\newblock Simple modules of Ariki-Koike algebras, 
\newblock{\em in Group representations: cohomology, group actions and topology, 
Proc. Sym. Pure Math.} {\bf 63} (1998), 383-396.  

  
\bibitem[M2]{Mat08}
A.~Mathas, 
\newblock{ Seminormal forms and {G}ram determinants for cellular algebras}, 
{\em J. Reine Angew. Math.} 
{\bf 619} 
(2008), 141-173

\bibitem[N]{N} 
T.~Nakashima, 
\newblock{ Crystal Base and a Generalization of the Littlewood-Richardson Rule for the Classical Lie Algebras}, 
 {\em Commun. Math. Phys.} {\bf 154} (1993), 215-243. 

\bibitem[SakS]{SakS}
M.~Sakamoto and T.~Shoji, 
\newblock{Schur-Weyl reciprocity for Ariki-Koike algebras}, 
\newblock{\em J. Algebra} {\bf 221} (1999), 293-314. 

\bibitem[Saw]{Saw}
N.~Sawada,   
\newblock On decomposition numbers of the cyclotomic {$q$}-{S}chur algebras, 
\newblock {\em J. Algebra} {\bf311} (2007), 147--177. 

\bibitem[SawS]{SS}
N.~Sawada and T.~Shoji, 
\newblock{Modified Ariki-Koike algebras and cyclotomic $q$-Schur algebras}, 
\newblock{\em Math. Z.} {\bf 249} (2005), 829-867.

\bibitem[Sho]{Sho}
T.~Shoji,  
\newblock{A Frobenius formula for the characters of Ariki-Koike algebras}, 
\newblock{\em J. Algebra} {\bf 226} (2000), 818-856.


\bibitem[SW]{SW} 
T.~Shoji and K.~Wada, 
\newblock Cyclotomic $q$-Schur algebras associated to the Ariki-Koike algebra, 
\newblock{\em Represent. Theory} {\bf 14} (2010), 379-416.


\bibitem[W]{W}
K.~Wada, 
\newblock{Presenting cyclotomic $q$-Schur algebras}, 
to appear in Nagoya Math. J. 
arXiv:0908.3306


\end{thebibliography}
\end{document}